\documentclass[11pt]{amsart}

\usepackage{amssymb}
\usepackage{amsmath}
\usepackage{amsthm}

\newcommand{\diag}{\mathop{\mathrm {diag}}\nolimits}

\newcommand{\tr}{\mathop{\mathrm {tr}}\nolimits}
\newcommand{\Ad}{\mathop{\mathrm {Ad}}\nolimits}

\newcommand{\Hom}{\mathop{\mathrm {Hom}}\nolimits}

\newcommand{\Ind}{\mathop{\mathrm {Ind}}\nolimits}

\newcommand{\id}{\mathop{\mathrm {id}}\nolimits}

\newcommand{\hatoplus}{\mathop{\widehat{\bigoplus}}}

\newcommand{\sI}{\sqrt{-1}}
\newcommand{\hs}{\hspace{0.5mm}}
\newcommand{\mhs}{\hspace{-2mm}}
\newcommand{\mmhs}{\hspace{-1mm}}

\newcommand{\Diag}{\mathop{\mathrm {Diag}}\nolimits}

\newcommand{\mC}{\mathbf{C}}

\newcommand{\mH}{\mathbf{H}}
\newcommand{\mP}{\mathbf{P}}
\newcommand{\mR}{\mathbf{R}}
\newcommand{\mS}{\mathbf{S}}
\newcommand{\mT}{\mathbf{T}}
\newcommand{\mZ}{\mathbf{Z}}
\newcommand{\me}{\mathbf{e}}

\newcommand{\mv}{\mathbf{v}}

\newcommand{\gC}{\mathfrak{C}}
\newcommand{\ga}{\mathfrak{a}}
\newcommand{\g }{\mathfrak{g}}

\newcommand{\gk}{\mathfrak{k}}
\newcommand{\gl}{\mathfrak{l}}
\newcommand{\gm}{\mathfrak{m}}
\newcommand{\gn}{\mathfrak{n}}
\newcommand{\go}{\mathfrak{o}}
\newcommand{\gp}{\mathfrak{p}}
\newcommand{\gs}{\mathfrak{s}}
\newcommand{\gu}{\mathfrak{u}}

\newtheorem{thm}{Theorem}[section]
\newtheorem{lem}[thm]{Lemma}
\newtheorem{prop}[thm]{Proposition}
\newtheorem{cor}[thm]{Corollary}
\theoremstyle{definition}
\newtheorem{defn}[thm]{Definition}

\numberwithin{equation}{section}
\theoremstyle{remark}

\begin{document}

\title{The structures of 
standard $(\g ,K)$-modules of $SL(3,\mR )$.}

\author[Tadashi Miyazaki]{Tadashi Miyazaki}
\address{Department of Mathematical Sciences, University of Tokyo}
\email{miyaza@ms.u-tokyo.ac.jp}


\maketitle
\begin{abstract}
We describe explicitly the structures of standard $(\g ,K)$-modules 
of $SL(3,\mR )$. 
\end{abstract}

\section{Introduction}
\label{sec:introduction}

As far as we know, 
for some `small' semisimple Lie groups $G$, the $(\g ,K)$-module 
structures of standard representations are completely described. 
For example, the description of them for $SL(2,\mR)$ 
is found in standard textbooks, and there are rather complete results 
for some groups of real rank $1$, e.g.\ 
$SU(n,1)$ in \cite{MR0330355} and $Spin(1,2n)$ in \cite{MR0453925}. 
However, for Lie groups of higher rank, there are few references 
as far as the author knows. 
It seems to be difficult to describe the whole 
$(\g ,K)$-module structures even for standard representations of 
classical groups of higher rank, since their $K$-types are not 
multiplicity free. In the papers \cite{pre_standard_1_2006} and 
\cite{pre_standard_2}, the $(\g ,K)$-module 
structures of some standard representations of $Sp(2,\mR )$ are 
described by T. Oda. In the former paper \cite{master_2_2006}, we extend 
the result for principal series representations of $Sp(3,\mR )$. 
The method in these papers is applicable 
to study of standard representations of another groups. 
In this paper, we use this method to study 
standard $(\g ,K)$-modules of $SL(3,\mR )$. 

Before describing the case of $SL(3,\mR )$, let us explain 
the problem in a more precise form 
for a general real semisimple Lie group $G$ with 
its Lie algebra $\g $. 
Fix a maximal compact subgroup $K$ of $G$. 
Since any standard $(\g ,K)$-modules are realized as subspaces of 
$L^2(K)$ as $K$-modules, we investigate the $K$-module structure of 
standard $(\g ,K)$-modules by the Peter-Weyl's theorem. 
In order to describe the action of $\g $ or $\g_\mC =\g \otimes_\mR \mC $, 
it suffices to investigate the action of $\gp $ or $\gp_\mC $, 
because of the Cartan decomposition $\g =\gk \oplus \gp $. 
Therefore, the investigation of the action of $\gp $ 
or $\gp_\mC $ is essential 
to give the description of the $(\g ,K)$-module structure of 
a standard representation. 
To study the action of $\gp_\mC $, we compute the linear map 
$\Gamma_{\tau ,i}$ defined as follows. 
Let $(\pi , H_\pi )$ be a standard representation of $G$ with its 
subspace $H_{\pi ,K}$ of $K$-finite vectors.
For  a $K$-type $(\tau ,V_\tau )$ of $\pi $, and 
a nonzero $K$-homomorphism $\eta \colon V_\lambda \to H_{\pi ,K}$, 
we define a linear map 
$\tilde{\eta }\colon \gp_\mC \otimes_\mC V_\lambda \to H_{\pi ,K}$ 
by $X\otimes v \mapsto X\cdot \eta (v) $. 
Then $\tilde{\eta }$ is a $K$-homomorphism with $\gp_\mC $ endowed with 
the adjoint action $\Ad $ of $K$. 
Let $V_\tau \otimes_\mC \gp_\mC \simeq \bigoplus_{i\in I}V_{\tau_i}$ 
be the decomposition into a direct sum of irreducible $K$-modules and 
$\iota_i$ an injective $K$-homomorphism from 
$V_{\tau_i}$ to $V_\tau \otimes_\mC \gp_\mC$ for each $i$. We define a 
linear map $\Gamma_{\tau ,i}\colon \Hom_K(V_\tau ,H_{\pi ,K})\to 
\Hom_K(V_{\tau_i},H_{\pi ,K})$ 
by $\eta \mapsto \tilde{\eta }\circ \iota_i$. 
These linear maps $\Gamma_{\tau ,i}\ (i\in I)$ characterize the action of 
$\gp_\mC $. Our purpose of this paper is to give explicit expressions 
of $\iota_i$ and $\Gamma_{\tau ,i}$ 
when $\pi $ is a $P$-principal series representation 
of $G=SL(3,\mR )$ for each standard parabolic subgroup $P$ of $G$. 
As a result, we obtain infinite number of 'contiguous relations', 
a kind of system of differential-difference relations among vectors in 
$H_{\pi}[\tau ]$ and $H_{\pi}[\tau_i]$. 
Here $H_{\pi}[\tau ] $ is $\tau$-isotypic component of $H_{\pi}$.
These are described in Proposition \ref{prop:injector}, 
Theorem \ref{th:main} and \ref{th:main2}. 

As an application, we can utilize the contiguous relations 
to obtain the explicit formulae of some spherical functions. 
In the paper \cite{MR2070374}, H. Manabe, T. Ishii and T. Oda 
give the explicit formulae of Whittaker functions of 
principal series representations of $SL(3,\mR )$ to solve 
the holonomic system of differential equations 
characterizing those functions, 
which is derived from the Capelli elements and 
the contiguous relations around minimal $K$-type. 
We can obtain the holonomic systems characterizing 
Whittaker functions of generalized 
principal series representations of $SL(3,\mR )$ 
from the result of this paper. 
We hope that this interesting possibility 
will be considered in future work. 
On the other hand, if we have the explicit formula 
of Whittaker function with a certain $K$-type, 
then we can give those with another $K$-type by using 
contiguous relations. 

We give the contents of this paper. 
In Section \ref{sec:review_sl2}, 
we recall the classical case $SL(2,\mR )$ shortly. 
In Section \ref{sec:preliminaries}, 
we recall the structure of $SL(3,\mR )$ and define a standard representations 
obtained by a parabolic induction with respect to 
the standard parabolic subgroups.
In Section \ref{sec:K-modules}, we introduce 
the standard basis of a finite dimensional 
irreducible representation of $K$ and 
give explicit expressions 
of $\iota_i\colon  V_{\tau_i}\to V_\tau \otimes_\mC \gp_\mC $. 
In Section \ref{sec:structure}, 
we introduce the general setting of this paper and 
give matrix representations of 
$\Gamma_{\tau ,i}$ for principal series representations 
in Theorem \ref{th:main}. 
In Section \ref{sec:structure2}, we give the matrix representations of 
$\Gamma_{\tau ,i}$ for generalized principal series representations 
in Theorem \ref{th:main2}. 
In Section \ref{sec:examples}, we give explicit expressions of 
the action of $\gp_\mC$ in Proposition \ref{prop:p_action}. 

\section*{Acknowledgments}
The author would like to express his gratitude 
to Takayuki Oda for valuable advice on this work 
and also thanks to Yasuko Hasegawa for correction 
of many typos.

\section{The standard $(\g ,K)$-modules of $SL(2,\mR )$}
\label{sec:review_sl2}
We start with a short review of the most classical case, i.e. 
the case of the group $SL(2,\mR )$.

\subsection{The principal series representations of $SL(2,\mR )$}
We denote by $\mZ $, $\mR $ and $\mC $ the ring of rational integers, 
the real number field and the complex number field, respectively. 
Let $\mZ_{\geq 0}$ be the set of non-negative integers, 
$1_n$ be the unit matrix in the space 
$M_n (\mR )$ of real matrices of size $n$ and 
$O_{m,n}$ be the zero matrix of size $m\times n$. 
We denote by $\delta_{ij}$ the Kronecker delta, 
i.e. 
\[
\delta_{ij}=
\left\{
\begin{array}{ll}
1,&i=j,\\
0,&\text{otherwise}.
\end{array}
\right.
\]
For a Lie algebra $\gl $, we denote by $\gl_\mC =\gl \otimes_\mR \mC $ 
the complexification of $\gl $.

We put
\begin{align*}
&G'=SL(2,\mR ),\ 
M'=\{ m=\diag (\varepsilon ,\varepsilon^{-1} )\mid 
\varepsilon \in \{\pm 1\}\},\ 
A'=\{ a(r)=\diag (r,r^{-1})\mid r\in \mR_{>0}\},\\ 
&N'=\left\{\left.
\left(\begin{array}{cc}
1&x\\
0&1
\end{array}\right)\right|x\in \mR\right\}
,\ K'=SO(2)=\left\{\left.\kappa_t=
\left(\begin{array}{cc}
\cos t&\sin t\\
-\sin t&\cos t
\end{array}\right)\right|t\in \mR \right\}.
\end{align*}
Let $\g ',\ \gk ',\ \ga '$ and $\gn '$ be Lie algebras of 
$G',\ K',\ A'$ and $N'$, respectively.

For $\nu \in \mC $ and a character $\sigma $ of $M'$, 
the principal series representation $\pi_{(\nu ,\sigma )}$ of $G'$ 
is defined as the right regular representation 
of $G'$ on the space $H_{(\nu ,\sigma )}$ which is the completion of 
\[
H_{(\nu ,\sigma )}^\infty=
\left\{ f\colon G'\to \mC \text{ smooth } \left| 
\begin{array}{c} 
f(namx)=r^{\nu +1 }\sigma (m) f(x) \hphantom{==========} \\
\text{ for }\ n\in N',\ a=a(r)\in A',\ m\in M',\ x\in G'
\end{array} \right. \right\}
\]
with respect to the norm
\[
\| f\|^2 
=\int_{K'}|f(k)|^2dk.
\]
The restriction map 
$r_{K'} \colon H_{(\nu ,\sigma )}\ni f\mapsto 
f|_{K'}\in L^2(K')$ 
is an injective $K'$-homomorphism when 
$L^2(K')$ is endowed with right regular action of $K'$. 
Then the image of $r_{K'}$ is the following subspace of $L^2(K')$: 
\[
L^2_{(M',\sigma )}(K')=\{ f\in L^2(K')\mid
f(mx)=\sigma (m)f(x)\ \text{for a.e. } m\in M',\ x\in K'\} .
\]

We have an irreducible decomposition of the $K'$-module $L^2(K')$:
\[
L^2(K')=\hatoplus_{p\in \mZ}\mC \cdot \tilde{\chi}_p,
\]
where $\tilde{\chi}_p\colon K'\ni \kappa_t\mapsto e^{\sI pt}\in 
\mC^\times$. 

Therefore we have an isomorphism
\[
H_{(\nu ,\sigma )}\to L^2_{(M',\sigma )}(K')=
\left\{\begin{array}{ll}
\hatoplus_{p\in 2\mZ}\mC \cdot \tilde{\chi}_p,&
\text{ if }\sigma (-1_2)=1,\\[2mm]
\hatoplus_{p\in 1+2\mZ}\mC \cdot \tilde{\chi}_p,&
\text{ if }\sigma (-1_2)=-1.
\end{array}\right.
\]
Let $\chi_p \in H_{(\nu ,\sigma )}$ be an inverse image of $\tilde{\chi}_p$ 
by this isomorphism. 

Now we take a basis $\{w,\ x_+,\ x_-\}$ of ${\g '}_\mC$ defined by 
\begin{align*}
w=&\left(\begin{array}{cc}
0&1\\
-1&0
\end{array}\right),&
x_\pm =&\left(\begin{array}{cc}
1&\pm \sI \\
\pm \sI &-1
\end{array}\right) .
\end{align*}
Here we note that  
\[
{\g '}_\mC ={\gk '}_\mC \oplus {\gp '}_\mC ,\quad 
{\gk '}_\mC =\mC \cdot w,\quad 
{\gp '}_\mC =\mC \cdot x_+\oplus \mC \cdot x_-.
\]
is a complexification of a Cartan decomposition 
${\g '}={\gk '}\oplus {\gp '}$ with respect to 
a Cartan involution $\g '\ni X\mapsto {-}^tX\in \g '$ 
where ${}^tX$ means transpose of $X$. 

Since $w\in \gk '$, we see that
\begin{equation}
\pi_{(\nu ,\sigma )}(w)\chi_p=\sI p\chi_p\label{eqn:review001}
\end{equation}
from direct computation.
Here we denote the differential of $\pi_{(\nu ,\sigma )}$ again 
by $\pi_{(\nu ,\sigma )}$. 
The action of ${\gp '}_\mC$ is given in 
the following proposition. 
\begin{prop}\label{prop:str_sl2}
$\pi_{(\nu ,\sigma )}(x_\pm )\chi_p=(\nu +1\pm p)\chi_{p\pm 2}$.
\end{prop}
\begin{proof}
By the relations 
\[
[w,x_{\pm }]=\pm 2\sI x_{\pm},
\]
we have
\begin{equation}
\pi_{(\nu ,\sigma )}(w)(\pi_{(\nu ,\sigma )}(x_\pm )\chi_p)
=\sI (p\pm 2)(\pi_{(\nu ,\sigma )}(x_\pm )\chi_p).\label{eqn:review002}
\end{equation}
Here $[\cdot , \cdot ]$ is the bracket product. 
From the equations (\ref{eqn:review001}) and (\ref{eqn:review002}), 
we see that $\pi_{(\nu ,\sigma )}(x_\pm )\chi_p\in \mC \cdot \chi_{p\pm 2}$.

The elements $x_\pm$ of ${\gp '}_\mC$ have 
the following expressions according to Iwasawa decomposition 
${\g '}_\mC ={\gn '}_\mC \oplus {\ga '}_\mC \oplus {\gk '}_\mC$:
\[
x_\pm =\pm 2\sI E' +H'\mp \sI w
\]
where $E'=\left(\begin{array}{cc}
0&1\\
0&0
\end{array}\right)\in {\gn '}_\mC$
 and $H'=\diag (1,-1)\in {\ga '}_\mC $. 
From this expression and the definition of the space $H_{(\nu ,\sigma )}$, 
we have the value of $\pi_{(\nu ,\sigma )}(x_\pm )\chi_p$ at 
$1_2=\kappa_0\in K'$ as follows:
\begin{align*}
\pi_{(\nu ,\sigma )}(x_\pm )\chi_p(1_2)=&
\pm 2\sI \pi_{(\nu ,\sigma )}(E')\chi_p(1_2)
+\pi_{(\nu ,\sigma )}(H')\chi_p(1_2)
\mp \sI \pi_{(\nu ,\sigma )}(w)\chi_p(1_2)\\
=&0+(\nu +1)\mp \sI (\sI p)\\
=&\nu +1\pm p.
\end{align*}
Since $\chi_{p\pm 2}(1_2)=1$, we obtain 
$\pi_{(\nu ,\sigma )}(x_\pm )\chi_p=(\nu +1\pm p)\chi_{p\pm 2}$.
\end{proof}

From this proposition, we obtain the following.
\begin{prop}
\label{prop:discrete}
(i) Let $k$ be an integer such that $k\geq 2$. 
If $\nu =k-1$ and $\sigma (-1)=(-1)^k$, 
there is an injective homomorphism 
from $D_k^\pm $ to $\pi_{(\nu ,\sigma )}$. 
Here $D_k^+$ and $D_k^- $ are discrete series representations of 
$SL(2,\mR )$ with the Blattner parameter $k$ and $-k\in \mZ$, respectively. 
Moreover the quotient 
$(\g ',K')$-modules $\pi_{(\nu ,\sigma )}/(D_k^+\oplus D_k^- )$ is 
of dimension $k-1$. \\
(ii) Let $k$ be an integer such that $k\geq 2$. 
If $\nu =-k+1$ and $\sigma (-1)=(-1)^k$, 
the $(k-1)$-dimensional subspace $F_{k-2}$ of $H_{(\nu ,\sigma )}$ 
generated by 
\[
\{\chi_p\mid p=-k+2,\ -k+4,\ \cdots k-2\}
\]
is $G'$-invariant and is isomorphic to the symmetric tensor 
representation of degree $k-2$. Moreover the quotient 
$\pi_{(\nu ,\sigma )}/F_{k-2}$ is isomorphic to 
$D_k^+\oplus D_k^-$. \\
(iii) If $\nu =0$ and $\sigma (1_2)=-1$, $\pi_{(\nu ,\sigma )}$ is 
a direct sum of two irreducible representations, 
called limit of discrete series representations. \\
(iv) If $(\nu ,\sigma )$ is not in the cases of (i), (ii) and (iii), 
$\pi_{(\nu ,\sigma )}$ is irreducible.
\end{prop}

We are going to show the analogue of Proposition 
\ref{prop:str_sl2} for $SL(3,\mR )$ in Theorem \ref{th:main} 
and \ref{th:main2}. 

\section{Preliminaries}\label{sec:preliminaries}
\subsection{Groups and algebras.}\label{subsec:groups_and_algebra}

Let $G$ be the special linear group $SL(3,\mR )$ of degree three and 
$\g$ be its Lie algebra.
We define a Cartan involution $\theta $ of $G$ by 
$G\ni g \mapsto {}^t g^{-1}\in G$. 
Here $g^{-1}$ means the inverse of $g$. 
Then a maximal compact subgroup of $G$ is given by
\[
K=\{ g \in G \mid \theta (g)=g\} =SO(3).
\]

If we denote the differential of $\theta $ again by $\theta $, then we have 
$\theta (X)={-}^t X$ for $X\in \g $. 
Let $\gk$ and $\gp$ be the $+1$ and the 
$-1$ eigenspaces of $\theta$ in $\g$, respectively, that is, 
\begin{align*}
\gk &=\{ X\in \g \mid {}^tX=-X\} =\gs \go (3),&\gp=&\{ X\in \g \mid {}^tX=X\} .
\end{align*}
Then $\gk $ is the Lie algebra of $K$ and 
$\g $ has the Cartan decomposition $\g =\gk \oplus \gp $.

Put $\ga_0 =\{ \diag (t_1,t_2,t_3)\mid t_i\in \mR \ (1\leq i\leq 3),\ 
t_1+t_2+t_3=0 \}$. 
Then $\ga_0 $ is a maximal abelian subalgebra of $\gp $. 
For each $1\leq i\leq 3$, we define a linear form $e_i$ on $\ga_0$ by 
$\ga_0 \ni \diag (t_1,t_2,t_3)\mapsto t_i \in \mC $.
The set $\Sigma $ of the restricted roots for $(\ga_0 ,\g )$ is given by 
$\Sigma =\Sigma (\ga_0 ,\g )=\{ e_i-e_j \mid 1\leq i\neq j\leq 3 \}$, 
and the subset $\Sigma^+ =\{ e_i-e_j \mid 1\leq i<j\leq 3 \} $ 
forms a positive root system. 
For each $\alpha \in \Sigma $, we denote the restricted root space by 
$\g_\alpha $ and choose a restricted root vector $E_\alpha $ in $\g_\alpha $ 
as follows:
\begin{align*}
E_{e_1-e_2}=&
\left( \begin{array}{ccc}
0 & 1 & 0 \\
0 & 0 & 0 \\
0 & 0 & 0 
\end{array} \right),&
E_{e_1-e_3}=&
\left( \begin{array}{ccc}
0 & 0 & 1 \\
0 & 0 & 0 \\
0 & 0 & 0 
\end{array} \right),&
E_{e_2-e_3}=&
\left( \begin{array}{ccc}
0 & 0 & 0 \\
0 & 0 & 1 \\
0 & 0 & 0 
\end{array} \right),
\end{align*}
and $E_{-\alpha }={}^tE_\alpha $ for $\alpha \in \Sigma^+ $. 
If we put $\gn_0 =\bigoplus_{\alpha \in \Sigma^+ } \g_\alpha $, then $\g $ has 
an Iwasawa decomposition $\g =\gn_0 \oplus \ga_0 \oplus \gk $. 
Also we have $G=N_0A_0K$, where $N_0=\exp (\gn_0 )$ 
and $A_0=\exp (\ga_0 )$.

The group $G$ has three non-trivial standard parabolic subgroups 
$P_0,\ P_1,\ P_2$ with 
\begin{align*}
P_0=&\left\{\left(
\begin{array}{ccc}
*&*&*\\
0&*&*\\
0&0&*
\end{array}
\right)\in G\right\}, &
P_1=&\left\{\left(
\begin{array}{ccc}
*&*&*\\
0&*&*\\
0&*&*
\end{array}
\right)\in G\right\}, &
P_2=&\left\{\left(
\begin{array}{ccc}
*&*&*\\
*&*&*\\
0&0&*
\end{array}
\right)\in G\right\}.
\end{align*} 
Let $\gn_1,\ \gn_2$ be subalgebras of $\gn_0$ defined by
$\gn_1=\g_{e_1-e_2}\oplus \g_{e_1-e_3},\ 
\gn_2=\g_{e_1-e_3}\oplus \g_{e_2-e_3}.$
We take a basis $\{ H_1,H_2\}$ of $\ga_0 $ defined by
\begin{align*}
H_1=&\left(\begin{array}{ccc}
 1 & 0 & 0 \\
 0 & 0 & 0 \\
 0 & 0 &-1
\end{array}\right),&
H_2=&\left(\begin{array}{ccc}
 0 & 0 & 0 \\
 0 & 1 & 0 \\
 0 & 0 &-1
\end{array}\right), 
\end{align*}
and set $H^{(1)}=2H_1-H_2,\ H^{(2)}=H_1+H_2$.
we define subalgebras $\ga_1,\ \ga_2$ of $\ga_0$ by 
$\ga_1=\mR \cdot H^{(1)},\ 
\ga_2=\mR \cdot H^{(2)}.$
We specify Langland decompositions of $P_i=N_iA_iM_i\ (0\leq i\leq 2)$ by
\begin{align*}
M_0&=
\{ \diag (\varepsilon_1, \varepsilon_2,\varepsilon_1\varepsilon_2) 
\mid \varepsilon_i \in \{ \pm 1\} \ (1\leq i\leq 2)\} ,\\
M_1&=
\left\{\left.
\left(\begin{array}{cc}
\det (h)^{-1}&O_{1,2}\\
O_{2,1}&h
\end{array}\right)
\right|
h\in SL^{\pm}(2,\mR )
\right\},
\quad A_1=\exp (\ga_1),\quad N_1=\exp (\gn_1),\\
M_2&=
\left\{\left.
\left(\begin{array}{cc}
h&O_{2,1}\\
O_{1,2}&\det (h)^{-1}
\end{array}\right)
\right|
h\in SL^{\pm}(2,\mR )
\right\},
\quad A_2=\exp (\ga_2),\quad N_2=\exp (\gn_2).
\end{align*}
Here 
$SL^{\pm}(2,\mR )=\{g\in GL(2,\mR )\mid \det (g)=\pm 1\}$.
For $i=1,2$, let $\gm_i$ be a Lie algebra of $M_i$.

\subsection{Definition of the $P_i$-principal series representations of $G$}
\label{subsec:principal-series}

For $0\leq i\leq 2$, in order to define the $P_i$-principal series 
representation of $G$, we prepare the data $(\nu_i,\sigma_i)$ as follows.

For $\nu_0 \in \Hom_\mR (\ga_0 ,\mC)$, we define a coordinate 
$(\nu_{0,1},\nu_{0,2}) \in \mC^2$ by $\nu_{0,i}=\nu_0(H_i)\ (i=1,2)$.
Then the half sum $\rho_0 
=\frac{1}{2} \left( \sum_{\alpha \in \Sigma_+} \alpha \right)=e_1-e_3$
of the positive roots has coordinate $(\rho_{0,1},\rho_{0,2})=(2,1)$.
We define a quasicharacter $e^{\nu_0} \colon A_0\to \mC^\times $ by
\[
e^{\nu_0} (a)
=a_1^{\nu_{0,1}}a_2^{\nu_{0,2}} ,\quad 
a=\diag (a_1,a_2,a_3)\in A_0.
\]
We fix a character $\sigma_0 $ of $M_0$. 
$\sigma_0 $ is realized by 
$(\sigma_{0,1},\sigma_{0,2})\in \{ 0,1\}^{\oplus 2} $ 
such that 
\[
\sigma_0 (\diag (\varepsilon_1, \varepsilon_2, \varepsilon_1\varepsilon_2))
= \varepsilon_1^{\sigma_{0,1}} \varepsilon_2^{\sigma_{0,2}},\quad 
\varepsilon_1,\varepsilon_2\in \{\pm 1\}.
\]

For each $i=1,2$, we identify $\nu_i\in \Hom_\mR (\ga_i ,\mC)$ 
with a complex number $\nu_i(H^{(i)})\in \mC$. 
Let $\rho_i\ (i=1,2)$ be the half sums of positive roots 
whose root spaces are contained in $\gn_i$, i.e. 
$\rho_1=\frac{1}{2}(2e_1-e_2-e_3),\ 
\rho_2=\frac{1}{2}(e_1+e_2-2e_3)$. 
Then both $\rho_1$ and $\rho_2$ are identified with 3. 
We identify $M_i\ (i=1,2)$ with $SL^{\pm}(2,\mR)$ 
by natural isomorphisms 
$m_i\colon SL^{\pm}(2,\mR )\to M_i\ (i=1,2)$ defined by
\begin{align*}
m_1(h)&=
\left(\begin{array}{cc}
\det (h)^{-1}&O_{1,2}\\
O_{2,1}&h
\end{array}\right),&
m_2(h)=&
\left(\begin{array}{cc}
h&O_{2,1}\\
O_{1,2}&\det (h)^{-1}
\end{array}\right) &(h\in SL^{\pm}(2,\mR )).
\end{align*}
Then we fix a discrete series representation 
$\sigma_i=D_k=
\Ind^{SL^{\pm}(2,\mR)}_{SL(2,\mR)} (D^+_k)$ of $M_i\simeq SL^{\pm}(2,\mR)$
where $D^+_k$ is a discrete series representation of $SL(2,\mR )$
with the Blattner parameter $k\geq 2$.

\begin{defn}
\textit 
For $0\leq i\leq 2$, we define the $P_i$-principal series representation 
$\pi_{(\nu_i,\sigma_i)}$ of $G$ by 
\[
\pi_{(\nu_i,\sigma_i)}=\Ind_{P_i}^G 
(1_{N_i}\otimes e^{\nu_i +\rho_i }\otimes \sigma_i ),
\]
i.e. $\pi_{(\nu_i,\sigma_i)}$ is the right regular representation 
of $G$ on the space $H_{(\nu_i,\sigma_i)}$ which is the completion of 
\[
H_{(\nu_i,\sigma_i)}^\infty=
\left\{ f\colon G\to V_{\sigma_i} \text{ smooth } \left| 
\begin{array}{c} 
f(namx)=e^{\nu_i +\rho_i } (a)\sigma_i (m)f(x) \hphantom{==} \\
\text{ for }\ n\in N_i,\ a\in A_i,\ m\in M_i,\ x\in G
\end{array} \right. \right\}
\]
with respect to the norm
\[
\| f\|^2 
=\int_K\| f(k)\|_{\sigma_i}^2dk.
\]
Here $V_{\sigma_i}$ is a representation space of $\sigma_i$ and 
$\| \cdot \|_{\sigma_i}$ is its norm.
\end{defn}

\section{Representations of $K=SO(3)$}
\label{sec:K-modules}
\subsection{The spinor covering}
\label{subsec:spin_cover}
To describe the finite dimensional representations of $SO(3)$, 
the simplest way seems to be the one utilizing the double covering 
$\varphi \colon SU(2)=Spin(3)\to SO(3)$. 
We use the following realization of the double covering $\varphi$, 
which is introduced in \cite{MR2070374}.

The Hamilton quaternion algebra $\mH$ is realized in $M_2(\mC )$ by 
\[
\mH =\biggl\{ 
\left( \begin{array}{cc}
a & b  \\
-\bar{b} & \bar{a}  
\end{array} \right) \in M_2(\mC )\ 
\biggl| \ 
a,b\in \mC
\biggl\}.
\]
Then $SU(2)$ is the subgroup of the multiplicative group consisting of 
quaternions with reduced norm $1$, i.e. $SU(2)=\{ x\in \mH \mid \det x=1\}$. 
Let $\mP =\{x\in \mH \mid \tr x=0\}$ be the $3$-dimensional real 
Euclidean space consisting of pure quaternions. 
Then for each $x\in SU(2)$, the map 
$\mP \ni p \mapsto x\cdot p\cdot x^{-1}\in \mP$ 
preserve the Euclidean norm $p\mapsto \det p$ and the orientation, hence 
we have the homomorphism
\[
\varphi \colon SU(2)\to SO(\mP ,\det )\simeq SO(3),
\]
which is surjective, since the range is a connected group. 
The kernel of this homomorphism is given by $\{\pm 1_2 \}$. 
An explicit expression of the covering map $\varphi$ is given by 
\begin{equation*}
\varphi (x)
=\left( \begin{array}{ccc}
p^2+q^2-r^2-s^2 & -2(ps-qr) & 2(pr+qs) \\
2(ps+qr) & p^2-q^2+r^2-s^2 & -2(pq-rs) \\
-2(pr-qs) & 2(pq+rs) & p^2-q^2-r^2+s^2 \\
\end{array} \right) 
\end{equation*}
for 
$x=\left( \begin{array}{cc}
p+\sI q & r+\sI s   \\
-r+\sI s & p-\sI q   
\end{array} \right) \in SU(2)\ (p,q,r,s\in \mR )$. 

By the derivation $d\varphi \colon \gs \gu (2)\to \gs \go (3)$ of 
$\varphi $, the standard generators: 
\begin{align*}
u_1=&\left( \begin{array}{cc}
\sI & 0 \\
0 &-\sI 
\end{array} \right) ,&
u_2=&\left( \begin{array}{cc}
 0 & 1 \\
-1 & 0 
\end{array} \right) ,&
u_3=&\left( \begin{array}{cc}
 0 & \sI \\
 \sI & 0 
\end{array} \right) 
\end{align*}
are mapped to $-2K_{23},\ 2K_{13}, -2K_{12}$ with 
\begin{align*}
K_{23}=&
\left( \begin{array}{ccc}
0 & 0 & 0 \\
0 & 0 & 1 \\
0 &-1 & 0 
\end{array} \right),&
K_{13}=&
\left( \begin{array}{ccc}
0 & 0 & 1 \\
0 & 0 & 0 \\
-1& 0 & 0 
\end{array} \right),&
K_{12}=&
\left( \begin{array}{ccc}
0 & 1 & 0 \\
-1& 0 & 0 \\
0 & 0 & 0 
\end{array} \right) \in \gs \go (3),
\end{align*}
respectively.

\subsection{Representations of $SU(2)$}
\label{subsec:rep_su2}

The set of equivalence classes of 
the finite dimensional continuous representations of $SU(2)$ is 
exhausted by the symmetric tensor product 
$\tau_l\ (l\in\mZ_{\geq 0})$ of the representation 
$SU(2)\ni g\mapsto (v\mapsto g\cdot v)\in GL(\mC^2)$. 
We use the following realizations of those which are introduced in 
\cite{MR2070374}.

Let $V_l$ be the subspace consisting of degree $l$ homogeneous polynomials 
of two variables $x,y$ in the polynomial ring $\mC [x,y]$. For $g\in SU(2)$ 
with $g^{-1}=\left( \begin{array}{cc}
a & b  \\
-\bar{b} & \bar{a}  
\end{array} \right) $ 
and $f(x,y)\in V_l$ we set 
\[
\tau_l(g)f(x,y)=f(ax+by,-\bar{b}x+\bar{a}y).
\]
Passing to the Lie algebra $\gs \gu (2)$, the derivation of $\tau_l$, 
denoted by same symbol, is described as follows by using the 
standard basis $\{v_k=x^ky^{l-k}\mid 0\leq k\leq l\}$ and 
the standard generators $\{ u_1,u_2,u_3\}$. 
Namely we have
\begin{align*}
\tau_l(H)v_k=&(l-2k)v_k,&\tau_l(E)v_k=&-kv_{k-1},&\tau_l(F)v_k=&(k-l)v_{k+1}.
\end{align*}
Here $\{E,\ H,\ F\}$ is $\gs \gl_2$-triple defined by 
\[
H=-\sI u_1,\ E=\frac{1}{2}(u_2-\sI u_3),\ 
F=-\frac{1}{2}(u_2+\sI u_3) \in \gs \gu (2)_\mC =\gs \gl (2,\mC ).
\] 

The condition that $\tau_l$ defines a representation of $SO(3)$ by passing 
to the quotient with respect to $\varphi \colon SU(2)\to SO(3)$ 
is that $\tau_l(-1_2)=(-1)^l=1$, i.e. $l$ is even. 
For $l\in \mZ_{\geq 0}$, we denote 
the irreducible representation of $SO(3)$ 
induced from $(\tau_{2l},V_{2l})$ again by 
$(\tau_{2l},V_{2l})$.

\subsection{The adjoint representation of $K$ on $\gp_\mC $}

It is known that $\gp_\mC $ becomes a $K$-module via 
the adjoint action of $K$. 
Concerning this, we have the following lemma. 
\begin{lem}\label{lem:ad_K_p}
\textit{ 
Let $\{ w_j\mid 0\leq j\leq 4\} $ be the standard basis of 
$(\tau_4,V_4)$ and 
$\{ X_j\mid 0\leq j\leq 4\} $ be a basis of $\gp_\mC $ 
defined as follows:
\begin{align*}
X_0=&
\left( \begin{array}{ccc}
0 & 0 & 0 \\
0 & 1 & -\sI \\
0 &-\sI & -1 
\end{array} \right),&
X_1=&
-\frac{1}{2}
\left( \begin{array}{ccc}
0 & \sI & 1 \\
\sI & 0 & 0 \\
1 & 0 & 0 
\end{array} \right),\\
X_2=&
-\frac{1}{3}
\left( \begin{array}{ccc}
2 & 0 & 0 \\
0 &-1 & 0 \\
0 & 0 &-1 
\end{array} \right) ,&
X_3=&
-\frac{1}{2}
\left( \begin{array}{ccc}
0 & \sI & -1 \\
\sI & 0 & 0 \\
-1 & 0 & 0 
\end{array} \right),\\
X_4=&
\left( \begin{array}{ccc}
0 & 0 & 0 \\
0 & 1 & \sI \\
0 &\sI & -1 
\end{array} \right) .
\end{align*}
Then via the unique isomorphism $V_4$ and $\gp_\mC$ as $K$-modules 
we have the identification $w_j=X_j\ (0\leq j\leq 4)$.
}
\end{lem}
\begin{proof}
By direct computation, 
we have the following table of the adjoint actions of the basis 
$\{d\varphi (E),\ d\varphi (H),\ d\varphi (F)\}$ of 
$\gk_\mC$ on the basis $\{X_j\mid 0\leq j\leq 4\}$ of $\gp_\mC$.
\begin{gather*}
\begin{array}{|c|c|c|c|c|c|}
\hline
&X_0&X_1&X_2&X_3&X_4\\ \hline
d\varphi (H)&4X_0&2X_1&0&-2X_3&-4X_4\\ \hline
d\varphi (E)&0&-X_0&-2X_1&-3X_2&-4X_3\\ \hline
d\varphi (F)&-4X_1&-3X_2&-2X_3&-1X_4&0\\ \hline
\end{array}
\\
\text{TABLE. The adjoint actions of $\gk_\mC$ on 
the basis $\{X_j\mid 0\leq j\leq 4\}$ of $\gp_\mC$.}
\end{gather*}
Comparing the actions in the above table with the actions 
in Subsection \ref{subsec:rep_su2}, we have the assertion. 
\end{proof}

\subsection{Clebsch-Gordan coefficients for the representations of 
$\gs \gl (2,\mC )$ with respect to standard basis}

In the later sections, we need irreducible decomposition of 
the tensor product $V\otimes_\mC \gp_\mC $ 
as $K$-modules for each $K$-type $(\tau ,V)$ of $\pi_{(\nu_i,\sigma_i)}$. 
From the previous arguments, it suffices to consider 
the irreducible decomposition of $V_l\otimes_\mC V_4$ 
as $\gs \gl (2,\mC )=\gs \gu (2)_\mC $-modules 
for arbitrary non-negative integer $l$. 

Generically, the tensor product $V_l\otimes_\mC V_4$ 
has five irreducible components 
$V_{l+4},\ V_{l+2},\ V_{l},\ V_{l-2}$ and $V_{l-4}$. 
Here some components may vanish. 
We give an explicit expression of a nonzero 
$\gs \gl (2,\mC )$-homomorphism from each irreducible component to 
$V_l\otimes_\mC V_4$ as follows.

\begin{prop}
\label{prop:injector}
\textit 
Let $\{v_{k}^{(l)}\mid 0\leq k\leq l\} $ be the standard basis of $V_l$ 
for $l\in \mZ_{\geq 0}$. 
We put $v_k^{(l)}=0$ when $k<0$ or $k>l$. 

If $V_{l+2m}$-component of $V_l\otimes_\mC V_4$ does not vanish, 
then we define linear maps 
$I^{l}_{2m}\colon V_{l+2m}\to V_l\otimes_\mC V_4\ (-2\leq m \leq 2)$ by 
\[
I^{l}_{2m}(v_k^{(l+2m)})
=\sum_{i=0}^{4}A_{[l,2m;k,i]}\cdot v_{k+2-m-i}^{(l)}\otimes w_i.
\]

Here the coefficients $A_{[l,2m;k,i]}=a(l,2m;k,i)/d(l,2m)$ are defined by 
following formulae. 

\noindent {\bf Formula 1:} 
The coefficients of 
$I^{l}_{4}\colon V_{l+4}\to V_l\otimes_\mC V_4$ 
are given as follows:
\begin{align*}
a(l,4;k,0)=&(l+4-k)(l+3-k)(l+2-k)(l+1-k),\\
a(l,4;k,1)=&4(l+4-k)(l+3-k)(l+2-k)k,\\
a(l,4;k,2)=&6(l+4-k)(l+3-k)k(k-1),\\
a(l,4;k,3)=&4(l+4-k)k(k-1)(k-2),\\
a(l,4;k,4)=&k(k-1)(k-2)(k-3),\\
d(l,4)=&(l+4)(l+3)(l+2)(l+1).
\end{align*}
\noindent {\bf Formula 2:} 
The coefficients of 
$I^{l}_{2}\colon V_{l+2}\to V_l\otimes_\mC V_4$ are given as follows:
\begin{align*}
a(l,2;k,0)=&(l+2-k)(l+1-k)(l-k),&
a(l,2;k,1)=&-(l+2-k)(l+1-k)(l-4k),\\
a(l,2;k,2)=&-3(l+2-k)(l-2k+2)k,&
a(l,2;k,3)=&-(3l-4k+8)k(k-1),\\
a(l,2;k,4)=&-k(k-1)(k-2),&
d(l,2)=&(l+2)(l+1)l.
\end{align*}
\noindent {\bf Formula 3:} 
The coefficients of 
$I^{l}_{0}\colon  V_{l}\to V_l\otimes_\mC V_4$ are given as follows:
\begin{align*}
a(l,0;k,0)=&(l-k)(l-1-k),&
a(l,0;k,1)=&-2(l-k)(l-2k-1),\\
a(l,0;k,2)=&(l^2-6kl+6k^2-l),&
a(l,0;k,3)=&2(l-2k+1)k,\\
a(l,0;k,4)=&k(k-1),&
d(l,0)=&l(l-1).
\end{align*}
\noindent {\bf Formula 4: }
The coefficients of 
$I^{l}_{-2}\colon  V_{l-2}\to V_l\otimes_\mC V_4$ are given as follows:
\begin{align*}
a(l,-2;k,0)=&(l-k-2),&
a(l,-2;k,1)=&-(3l-4k-6),&
a(l,-2;k,2)=&3(l-2k-2),\\
a(l,-2;k,3)=&-(l-4k-2),&
a(l,-2;k,4)=&-k,&
d(l,-2)=&l-2.
\end{align*}
\noindent {\bf Formula 5:} 
The coefficients of 
$I^{l}_{-4}\colon  V_{l-4}\to V_l\otimes_\mC V_4$ are given as follows:
\begin{align*}
a(l,-4;k,0)=&1,&
a(l,-4;k,1)=&-4,&
a(l,-4;k,2)=&6,\\
a(l,-4;k,3)=&-4,&
a(l,-4;k,4)=&1,& 
d(l,-4)=&1.
\end{align*}

Then $I^{l}_{2m}$ is a generator of 
$\Hom_{\gs \gl (2,\mC )} (V_{l+2m},V_l\otimes_\mC V_4)$, 
which is unique up to scalar multiple. 
\end{prop}
\begin{proof}
We have
\begin{align*}
&(\tau_{l}\otimes \tau_{4})(E)\circ I^{l}_{2m}(v_0^{(l+2m)})\\
&=\sum_{i=0}^{4}A_{[l,2m;0,i]}\cdot 
(\tau_{l}(E)v_{2-m-i}^{(l)})\otimes w_i
+\sum_{i=0}^{4}A_{[l,2m;0,i]}\cdot 
v_{2-m-i}^{(l)}\otimes (\tau_{4}(E)w_i)\\
&=\sum_{i=0}^{4}A_{[l,2m;0,i]}\cdot 
(-(2-m-i)v_{1-m-i}^{(l)})\otimes w_i
+\sum_{i=1}^{4}A_{[l,2m;0,i]}\cdot 
v_{2-m-i}^{(l)}\otimes (-iw_{i-1})\\
&=-\sum_{i=0}^{4}
((2-m-i)A_{[l,2m;0,i]}+(i+1)A_{[l,2m;0,i+1]})\cdot 
v_{1-m-i}^{(l)}\otimes w_i.
\end{align*}
Here we put $A_{[l,2m;0,5]}=0$.
By direct computation, we confirm 
\[
(2-m-i)A_{[l,2m;0,i]}+(i+1)A_{[l,2m;0,i+1]}=0
\]
for $-2\leq m\leq 2$ and $0\leq i\leq 4$. Hence 
\[
(\tau_{l}\otimes \tau_{4})(E)\circ I^{l}_{2m}(v_0^{(l+2m)})=0.
\]
Moreover, we have
\[
(\tau_{l}\otimes \tau_{4})(H)\circ I^{l}_{2m}(v_0^{(l+2m)})
=(l+2m)I^{l}_{2m}(v_0^{(l+2m)}),
\]
since
\begin{align*}
(\tau_{l}\otimes \tau_{4})(H)(v_i^{(l)}\otimes w_j)
&=(\tau_l(H)v_i^{(l)})\otimes w_j+v_i^{(l)}\otimes (\tau_4(H)w_j)\\
&=(l+4-2i-2j)v_i^{(l)}\otimes w_j.
\end{align*}
This means $I^{l}_{2m}(v_0^{(l+2m)})$ is the highest weight vector 
of $V_{l+2m}$-component of $V_l\otimes_\mC V_4$ 
with respect to a Borel subalgebra $(\mC \cdot H)\oplus (\mC \cdot E)$ of 
$\gs \gl (2,\mC )$.  

Therefore, in order to complete the proof, it suffices to confirm 
\[
(\tau_{l}\otimes \tau_{4})(F)\circ I^{l}_{2m}(v_k^{(l+2m)})
=I^{l}_{2m}\circ \tau_{l+2m}(F)(v_k^{(l+2m)})
\]
for each $0\leq k\leq l+2m$.

We confirm these equations by direct computation. 
\end{proof}

The coefficients $A_{[l,2m;k,i]}$ in the above proposition 
have the following relations.
\begin{lem}
\label{lem:rel_clebsch}
The coefficients $A_{[l,2m;k,i]}$ in Proposition \ref{prop:injector} 
satisfy following relations:
\begin{align*}
&A_{[l,2m;l+2m-k,0]}
=(-1)^mA_{[l,2m;k,4]},\quad A_{[l,2m;l+2m-k,2]}
=(-1)^mA_{[l,2m;k,2]},\\
&3\{(k-m+1)A_{[l,2m;k,1]}
+(l-k+m+1)A_{[l,2m;k,3]}\}
=(ml+m^2+m-6)A_{[l,2m;k,2]}.
\end{align*}
for $-2\leq m\leq 2$ and $0\leq k\leq l+2m$. 
\end{lem}
\begin{proof}
These are obtained by direct computation.
\end{proof}

\subsection{The dual representation of $(\tau_l,V_l)$}
We denote by $(\tau^* ,V^*)$ the dual representation of $(\tau ,V)$. 
Here we note that $V_l^*$ is equivalent to $V_l$ as $SU(2)$-modules, 
since irreducible $l+1$-dimensional representation of $SU(2)$ 
is unique up to isomorphism. 

\begin{lem}
\label{lem:dual_Krep}
Let $\{v_k^{(l)*}\mid 0\leq k\leq l\} $ is the dual basis of 
the standard basis $\{v_k^{(l)}\mid 0\leq k\leq l\} $.
Via the unique isomorphism between $V_l$ and $V_l^*$ as $K$-modules 
we have the identification
\[
v_k^{(l)}=(-1)^k\frac{(l-k)!k!}{l!}v_{l-k}^{(l)*}
\]
for $0\leq k\leq l$.
\end{lem}
\begin{proof}
We denote by $\langle ,\rangle$ the canonical pairing on 
$V_{l}^*\otimes_\mC V_{l}$. 

Since
\[
\langle \tau_l^*(H)v_{k}^{(l)*},v_{m}^{(l)}\rangle 
=-\langle v_{k}^{(l)*},\tau_l(H)v_{m}^{(l)}\rangle 
=(2m-l)\delta_{km}=(2k-l)\delta_{km}, 
\]
we have $\tau_l^*(H)v_{k}^{(l)*}=(2k-l)v_{k}^{(l)*}$.
Similarly, we obtain
\begin{align*}
\tau_l^*(E)v_{k}^{(l)*}&=(k+1)v_{k+1}^{(l)*},&
\tau_l^*(F)v_{k}^{(l)*}&=(l-k+1)v_{k-1}^{(l)*}.
\end{align*}
From these equations, we obtain the assertion.
\end{proof}

\section{The $(\g ,K)$-module structures of 
principal series representations}
\label{sec:structure}

\subsection{Irreducible decomposition of 
$(\pi_{(\nu_0,\sigma_0)}|_K,H_{(\nu_0,\sigma_0)})$ 
as $K$-modules}
\label{subsec:peter-weyl0}

We set 
\[
L^2_{(M_0,\sigma_0 )}(K)=\{ f\in L^2(K)\mid
f(mx)=\sigma_0 (m)f(x)\ \text{for a.e. } m\in M,\ x\in K\} 
\]
and give a $K$-module structure by the right regular action of $K$. 
Then the restriction map 
$r_K \colon H_{(\nu_0,\sigma_0)}\ni f\mapsto 
f|_K\in L^2_{(M_0,\sigma_0)}(K)$ 
is an isomorphism of $K$-modules. 

$L^2(K)$ has a  $K\times K$-bimodule structure 
by the two sided regular action:
\[
((k_1,k_2)f)(x)=f(k_1^{-1}xk_2),\quad 
x\in K,\ f\in L^2(K),\ (k_1,k_2)\in K\times K.
\]
Then we define a homomorphism $\Phi_l\colon 
V_{2l}^{*}\otimes_\mC V_{2l} \to L^2(K)$ 
of $K\times K$-bimodules by 
\[
w\otimes v\mapsto (x\mapsto 
\langle w,\tau_{2l} (x)v\rangle ).
\] 
Then the Peter-Weyl's theorem tells that 
\[
\hatoplus_{l \in \mZ_{\geq 0}} \Phi_l\colon 
\hatoplus_{l \in \mZ_{\geq 0}} V_{2l}^{*}\otimes_\mC V_{2l} \to L^2(K)
\] 
is an isomorphism as $K\times K$-bimodules. 
Here $\hatoplus $ means a Hilbert space direct sum.

Since $L^2_{(M_0,\sigma_0 )}(K)\subset L^2(K)$, 
we have an irreducible decomposition of 
$L^2_{(M_0,\sigma_0 )}(K)$: 
\[
L^2_{(M_0,\sigma_0 )}(K)\simeq \hatoplus_{l \in \mZ_{\geq 0}} 
(V_{2l}^{*}[\sigma_0 ])\otimes_\mC V_{2l} .
\]
Here $V[\sigma_0 ]$ means the $\sigma_0 $-isotypic component in 
$(\tau|_{M_0},V)$ for a $K$-module $(\tau ,V)$.
Therefore we obtain an isomorphism 
\[
r_K^{-1}\circ \hatoplus_{l \in \mZ_{\geq 0}} \Phi_l\colon 
\hatoplus_{l \in \mZ_{\geq 0}} 
(V_{2l}^{*}[\sigma_0 ])\otimes_\mC V_{2l} \to H_{(\nu_0,\sigma_0)}.
\]

Since $M_0$ is generated by the two elements
\begin{align*}
m_{0,1}=&
\left(\begin{array}{ccc}
-1 & 0 & 0 \\
 0 & 1 & 0 \\
 0 & 0 &-1 
\end{array}\right),&
m_{0,2}=&
\left(\begin{array}{ccc}
 1 & 0 & 0 \\
 0 &-1 & 0 \\
 0 & 0 &-1 
\end{array}\right) \in M_0 ,
\end{align*}
we note that 
$v\in V_{2l}[\sigma_0 ]$ if and only if 
\begin{align*}
\tau_{2l}(m_{0,i})v=&\sigma_0 (m_{0,i})v=(-1)^{\sigma_{0,i}}v&
(i=1,2)
\end{align*}
for $v\in V_{2l}$.
From the definition of $(\tau_{2l},V_{2l})$ and 
\begin{align*}
\varphi^{-1}_1(m_{0,1})=&\left\{ 
\pm \left(\begin{array}{cc}
 0 & 1 \\
-1 & 0
\end{array}\right)
\right\} ,&
\varphi^{-1}_1(m_{0,2})=&\left\{ 
\pm \left(\begin{array}{cc}
\sI& 0 \\
 0 &-\sI
\end{array}\right)
\right\} ,
\end{align*}
we have $\tau_{2l}(m_{0,1})v^{(2l)}_k=(-1)^{k}v^{(2l)}_{2l-k}$ and 
$\tau_{2l}(m_{0,2})v^{(2l)}_k=(-1)^{l-k}v^{(2l)}_{k}$. 
Hence we have 
\begin{align*}
V_{2l}[\sigma_0 ]=&\bigoplus_{k\in Z(\sigma_0 ;l)}\mC \cdot 
(v^{(2l)}_{2l-k}+(-1)^{\varepsilon (\sigma_0 ;l)}v^{(2l)}_{k}),
\end{align*}
where $\varepsilon (\sigma_0 ;l)\in \{ 0, 1\}$ such that 
${\varepsilon (\sigma_0 ;l)}\equiv l-\sigma_1-\sigma_2\bmod 2$ 
and 
\[
Z(\sigma_0 ;l)=
\left\{\begin{array}{ll}
\{ k\in \mZ \mid 0\leq k\leq l,\ k\equiv l-\sigma_{0,2} \bmod 2\}&
\text{if } \varepsilon (\sigma_0 ;l)=0,\\
\{ k\in \mZ \mid 0\leq k\leq l-1,\ k\equiv l-\sigma_{0,2} \bmod 2\}&
\text{if } \varepsilon (\sigma_0 ;l)=1.
\end{array}
\right.
\]
By the identification $V_{2l}^{*}=V_{2l}$ in 
Lemma \ref{lem:dual_Krep}, we note that 
$\{ v^{(2l)*}_{2l-k}+(-1)^{\varepsilon (\sigma_0 ;l)}v^{(2l)*}_{k}
\mid k\in Z(\sigma_0 ;l)\}$ is the basis of $V_{2l}^{*}[\sigma_0 ]$.

Now we define the elementary function 
$s(l;p,q)\in H_{(\nu_0,\sigma_0)}$ by 
\[
s(l;p,q)=r_K^{-1}\circ \Phi_l^{(j)}
((v^{(2l)*}_{2l-p}+(-1)^{\varepsilon (\sigma_0 ;l)}v^{(2l)*}_{p})\otimes 
v^{(2l)}_{q})
\]
for $l\in \mZ_{\geq 0},\ p\in Z(\sigma_0 ;l)$ and $0\leq q\leq 2l$. 

For each $p\in Z(\sigma_0 ;l)$, we put 
$S(l;p)$ a column vector of degree 
$2l+1$ whose $q+1$-th component is $s(l;p,q)$, i.e. 
${}^t(\ s(l;p,0),\ s(l;p,1),\ \cdots \ ,s(l;p,2l)\ ).$

Moreover we denote by $\langle S(l;p)\rangle $ 
the subspace of $H_{(\nu_0,\sigma_0)}$ generated by 
the functions in the entries of the vector $S(l;p)$, i.e. 
$\langle S(l;p)\rangle =
\bigoplus_{q=0}^{2l}\mC \cdot s(l;p,q)\simeq V_{2l}$. 
Via the unique isomorphism between $\langle S(l;p)\rangle$ 
and $V_{2l}$, we identify $\{s(l;p,q)\mid 0\leq q\leq 2l\}$ 
with the standard basis. 

From above arguments, we obtain the following.
\begin{prop}
As an unitary representation of $K$, it has an irreducible 
decomposition:
\[
H_{(\nu_0,\sigma_0)}
\simeq \hatoplus_{l\in \mZ_{\geq 0}} 
(V_{2l}^{*}[\sigma_0 ])\otimes_\mC V_{2l}. 
\]
Then the $\tau_{2l}$-isotypic component of $\pi_{(\nu_0,\sigma_0)}$ is 
given by 
\[
\bigoplus_{p\in Z(\sigma_0 ;l)}
\langle S(l;p) \rangle .
\]
\end{prop}
\begin{cor}
Let $d(\sigma_0 ;l)$ be the dimension of the space 
$\Hom_K(V_{2l} ,H_{(\nu_0,\sigma_0),K})$ of intertwining operators. 
Then
\[
d(\sigma_0 ;l)=\left\{
\begin{array}{ll}
(l+2)/2&\text{if $(\sigma_{0,1}, \sigma_{0,2})=(0,0)$ and $l$ is even,}\\
(l-1)/2&\text{if $(\sigma_{0,1}, \sigma_{0,2})=(0,0)$ and $l$ is odd,}\\
l/2&\text{if $(\sigma_{0,1}, \sigma_{0,2})\neq (0,0)$ and $l$ is even,}\\
(l+1)/2&\text{if $(\sigma_{0,1}, \sigma_{0,2})\neq (0,0)$ and $l$ is odd.}\\
\end{array}
\right.
\]
\end{cor}

\subsection{General setting }
\label{subsec:setting}
Let $H_{(\nu_i,\sigma_i),K}$ be the $K$-finite part of 
$H_{(\nu_i,\sigma_i)}$. 
In order to describe the action of $\g $ or $\g_\mC =\g \otimes_\mR \mC $, 
it suffices to investigate the action of $\gp $ or $\gp_\mC $, 
because of the Cartan decomposition $\g =\gk \oplus \gp $. 

For a $K$-type $(\tau_{2l} ,V_{2l} )$ of 
$\pi_{(\nu_i,\sigma_i)}$ and 
a nonzero $K$-homomorphism $\eta \colon 
V_{2l} \to H_{(\nu_i,\sigma_i),K}$,
we define a linear map 
\[
\tilde{\eta }\colon \gp_\mC \otimes_\mC V_{2l} 
\to H_{(\nu_i,\sigma_i),K}
\]
by $X\otimes v \mapsto \pi_{(\nu_i,\sigma_i)}(X)\eta (v) $. 
Here we denote differential of $\pi_{(\nu_i,\sigma_i)}$ again by 
$\pi_{(\nu_i,\sigma_i)}$. 
Then $\tilde{\eta }$ is $K$-homomorphism with $\gp_\mC $ endowed with 
the adjoint action $\Ad$ of $K$. 

Since 
\begin{align*}
&V_{2l} \otimes_\mC  \gp_\mC \simeq 
V_{2l} \otimes_\mC  V_{4} \simeq 
\bigoplus_{-2\leq m\leq 2} V_{2(l+m)},
\end{align*}
there are five injective $K$-homomorphisms 
\begin{align*}
I^{2l}_{2m}\colon &
V_{2(l+m)} \to V_{2l} \otimes_\mC \gp_\mC ,&
-2\leq m\leq 2 
\end{align*}
for general $l\in \mZ_{\geq 0}$. 
Then we define $\mC $-linear maps 
\[
\Gamma_{l,m}^i \colon 
\Hom_K(V_{2l} ,H_{(\nu_i,\sigma_i),K})
\to \Hom_K(V_{2(l+m)},H_{(\nu_i,\sigma_i),K}),\quad 
-2\leq m\leq 2
\]
by $\eta \mapsto \tilde{\eta }\circ I^{2l}_{2m}$.

Now we settle two purposes of this paper:
\begin{description}
\item[(i)] Describe the injective $K$-homomorphism $I^{2l}_{2m}$ 
in terms of the standard basis.
\item[(ii)] Determine the matrix representations of 
the linear homomorphisms $\Gamma_{l,m}^i$ 
with respect to the induced basis defined in the next subsection. 
\end{description}
We have already accomplished the first purpose in 
Proposition \ref{prop:injector}. 
We accomplish the second purpose 
in Theorem \ref{th:main} and \ref{th:main2}. 
As a result, we obtain infinite number of 'contiguous relations', a kind 
of system of differential-difference relations among vectors in 
$H_{(\nu_i,\sigma_i)}[\tau_{2l}]$ and 
$H_{(\nu_i,\sigma_i)}[\tau_{2(l+m)}]$. 
Here $H_{(\nu_i,\sigma_i)}[\tau ] $ is $\tau$-isotypic component of 
$H_{(\nu_i,\sigma_i)}$.

\subsection{The canonical blocks of elementary functions}
\label{subsec:canonical_blocks}

Let $\eta \colon V_{2l} \to H_{(\nu_i,\sigma_i),K}$ be a non-zero 
$K$-homomorphism. 
Then we identify $\eta $ with the column vector of degree 
$2l+1$ whose $q+1$-th component is $\eta (v^{(2l)}_q)$ 
for $0\leq q\leq 2l$, i.e. 
${}^t(\ \eta (v^{(2l)}_0),\ \eta (v^{(2l)}_1),\ \cdots \ ,
\eta (v^{(2l)}_{2l})\ ).$

By this identification, 
we identify $S(l;p)$ with
the injective $K$-homomorphism 
\[
V_{2l} \ni v^{(2l)}_q\mapsto s(l;p,q)\in H_{(\nu_0,\sigma_0),K},\quad 
0\leq q\leq 2l
\]
for $p\in Z(\sigma_0 ;l)$. 
We note that $\{ S(l;p) \mid p\in Z(\sigma_0 ;l)\} $ 
is a basis of $\Hom_K (V_{2l},H_{(\nu_0,\sigma_0),K})$ and we call 
it \textit{the induced basis from the standard basis}.

We define a certain matrix of elementary functions corresponding to 
the induced basis $\{ S(l;p) \mid p\in Z(\sigma_0 ;l)\}$ 
of $\Hom_K (V_{2l},H_{(\nu_0,\sigma_0),K})$
for each $K$-type $\tau_{2l}$ of 
our principal series representation $\pi_{(\nu_0,\sigma_0)}$.
\begin{defn}
\textit{
The following $(2l+1)\times d(\sigma_0 ;l)$ matrix $\mS (\sigma_0 ;l)$ 
is called \textit{the canonical block of elementary functions} 
for $\tau_{2l}$-isotypic component:
When $(\sigma_{0,1} ,\sigma_{0,2})=(0,0)$, we consider the matrix
\begin{align*}
\mS (\sigma_0 ;l)=&\left\{
\begin{array}{ll}
(\ S(l;0),\ S(l;2),\ S(l;4),\ \cdots \ ,S(l;l)\ )&
\text{if $l$ is even,}\\
(\ S(l;1),\ S(l;3),\ S(l;5),\ \cdots \ ,S(l;l-2)\ )&
\text{if $l$ is odd.}
\end{array}
\right.
\end{align*}
When $(\sigma_{0,1} ,\sigma_{0,2})=(1,0)$, we consider the matrix
\begin{align*}
\mS (\sigma_0 ;l)=&\left\{
\begin{array}{ll}
(\ S(l;0),\ S(l;2),\ S(l;4),\ \cdots \ ,S(l;l-2)\ )&
\text{if $l$ is even,}\\
(\ S(l;1),\ S(l;3),\ S(l;5),\ \cdots \ ,S(l;l)\ )&
\text{if $l$ is odd.}
\end{array}
\right.
\end{align*}
When $\sigma_{0,2}=1$, we consider the matrix
\begin{align*}
\mS (\sigma_0 ;l)=&\left\{
\begin{array}{ll}
(\ S(l;1),\ S(l;3),\ S(l;5),\ \cdots \ ,S(l;l-1)\ )&
\text{if $l$ is even,}\\
(\ S(l;0),\ S(l;2),\ S(l;4),\ \cdots \ ,S(l;l-1)\ )&
\text{if $l$ is odd.}
\end{array}
\right.
\end{align*}
}
\end{defn}

\subsection{The $\gp_\mC $-matrix corresponding to $I^{2l}_{2m}$}
\label{subsec:p-matrix}
For two integers $c_0,\ c_1$ such that $c_0\leq c_1$ and 
a rational function $f(x)$ in the variable $x$, we denote by 
\[
\underset{c_0\leq n\leq c_1}{\Diag}(f(n))
\]
the diagonal matrix of size $c_1-c_0+1$ with an entry $f(n)$ at the 
$(n-c_0+1,n-c_0+1)$-th component. 
Let $\me_i^{(l)}\ (0\leq i\leq l)$ be the column unit vector 
of degree $l+1$ with its $i+1$-th component $1$ and 
the remaining components 0. Moreover, let $\me_i^{(l)}$ be the 
column zero vector of degree $l+1$ when $i<0$ or $l<i$. 

In this subsection, we define $\gp_\mC $-matrix 
$\gC_{l,m}$ of size $(2(l+m)+1)\times (2l+1) $ 
corresponding to $I^{2l}_{2m}$ 
with respect to the standard basis. 

Let $\sum_{i=0}^4\iota_i^{(l,m)}\otimes X_i$ be 
the image of $I^{2l}_{2m}$ by 
the composite of natural linear maps
\begin{align*}
\Hom_K (V_{2(l+m)},V_{2l}\otimes_\mC \gp_\mC )
&\to \Hom_\mC (V_{2(l+m)},V_{2l}\otimes_\mC \gp_\mC )
\simeq \Hom_\mC (V_{2(l+m)},V_{2l})\otimes_\mC \gp_\mC .
\end{align*}
Then we define $\gp_\mC $-matrix 
$\gC_{l,m}=\sum_{i=0}^4R(\iota_i^{(l,m)})\otimes X_i$
where $R(\iota_i^{(l,m)})$ is the matrix representation of $\iota_i^{(l,m)}$
with respect to the standard basis. 
Explicit expression of the matrix $R(\iota_i^{(l,m)})$ of size 
$(2(l+m)+1)\times (2l+1)$ is given by
\begin{align*}
&\left(O_{2(l+m)+1,m+2},R(\iota_0^{(l,m)}),O_{2(l+m)+1,m+2}\right)\\
&=\left( O_{2(l+m)+1,4-i},\ 
\underset{0\leq k\leq 2(l+m)}{\Diag}(A_{[2l,2m;k,i]})
,\ O_{2(l+m)+1,i}\right)
\end{align*}
for $-2\leq m\leq 2$ and $0\leq i\leq 4$.
Here we erase the symbol $O_{m,n}$ when $m=0$ or $n=0$. 

For a column vector $\mv ={}^t(v_0,v_1,\cdots ,v_{2l})\in 
(H_{(\nu_i,\sigma_i),K})^{\oplus 2l+1}$ which is identified 
with an element of $\Hom_K(V_{2l},H_{(\nu_i,\sigma_i),K})$, 
we define $\gC_{l,m}\mv \in (H_{(\nu_i,\sigma_i),K})^{\oplus 2(l+m)+1}
\simeq \mC^{2(l+m)+1}\otimes_\mC H_{(\nu_0,\sigma_0),K}$ by
\begin{align*}
\gC_{l,m}\mv =\underset{0\leq q\leq 2l}{\sum_{0\leq i\leq 4}}
(R(\iota_i^{(l,m)})\cdot \me_{q}^{(2l)})\otimes 
(\pi_{(\nu_i,\sigma_i)}(X_i)v_q).
\end{align*}
Here $R(\iota_i^{(l,m)})\cdot \me_{q}^{(2l)}$ is the ordinal product of 
matrices $R(\iota_i^{(l,m)})$ and $\me_{q}^{(2l)}$.

From the definition of $\gC_{l,m}$, 
we note that the vector $\gC_{l,m}\mv $ 
is identified with the image of $\mv $ by 
$\Gamma_{l,m}^i$. 


\subsection{The contiguous relations}
\label{subsec:matrix_rep}

\begin{lem}
\label{lem:Iwasawa}
\textit{
The standard basis $X_i\ (0\leq i\leq 4)$ in $\gp_\mC$ 
have the following expressions according to the Iwasawa decomposition 
$\g_\mC =\gn_\mC \oplus \ga_\mC \oplus \gk_\mC $:
\begin{align*}
X_0=&-2\sI E_{e_2-e_3}+H_2+\sI K_{23},\\
X_1=&-(E_{e_1-e_3}+\sI E_{e_1-e_2})
	+\frac{1}{2}(K_{13}+\sI K_{12}),\\
X_2=&-\frac{1}{3}(2H_1-H_2),\\
X_3=&(E_{e_1-e_3}-\sI E_{e_1-e_2})
-\frac{1}{2}(K_{13}-\sI K_{12}),\\
X_4=&2\sI E_{e_2-e_3}+H_2-\sI K_{23}.
\end{align*}
}
\end{lem}
\begin{proof}
We obtain the assertion immediately from Lemma \ref{lem:ad_K_p}. 
\end{proof}

We give the matrix representation of $\Gamma_{l,m}^0$ 
with respect to the induced basis as follows.

\begin{thm}
\label{th:main}
For $l\in \mZ_{\geq 0},\ -2\leq m\leq 2$ 
such that $d(\sigma_0 ;l)>0$ and $d(\sigma_0 ;l+m)>0$, 
we have 
\begin{equation}
\gC_{l,m}\mS (\sigma_0 ;l)
=\mS (\sigma_0 ;l+m)\cdot R(\Gamma_{l,m}^0)
\label{eqn:statement_thm}
\end{equation}
with the matrix representation 
$R(\Gamma_{l,m}^0)\in M_{d(\sigma_0 ;l+m),d(\sigma_0 ;l)}(\mC )$ 
of $\Gamma_{l,m}^0$ with respect to the induced basis 
$\{ S(l;p) \mid p\in Z(\sigma_0 ;l)\} $: 

Explicit expressions of the matrix $R(\Gamma_{l,m}^0)$ of size 
$d(\sigma_0 ; l+m)\times d(\sigma_0 ;l)$ is given as follows:

When $\sigma_{0,2}=0 $ and $(m,\sigma_{0,1}+l)\in \{0,\pm 2\}\times (2\mZ)$,
the matrix $R(\Gamma_{l,m}^0)$ is given by
\begin{align*}
\left(\begin{array}{c}
\mmhs O_{n(\sigma_0 ;l,m),d(\sigma_0 ;l)}\mmhs \\[1mm]
R(\Gamma_{l,m}^0)
\end{array}\right) 
=&\left(\begin{array}{c}
\mmhs \underset{0\leq k\leq d(\sigma_0 ;l)-1}{\Diag}
\mhs \left(\gamma_{[l,m;2k+\delta (\sigma_0 ;l),-1]}^{(0)}\right)\mhs \\[1mm]
O_{1,d(\sigma_0 ;l)}
\end{array}\right)
+\left(\begin{array}{c}
O_{1,d(\sigma_0 ;l)}\\[1mm]
\mmhs \underset{0\leq k\leq d(\sigma_0 ;l)-1}{\Diag}
\mhs \left(\gamma_{[l,m;2k+\delta (\sigma_0 ;l),0]}^{(0)}\right)\mhs 
\end{array}\right)\\
&+\left(\begin{array}{cc}
O_{2,d(\sigma_0 ;l)-1}&O_{2,1}\\[1mm]
\mmhs \underset{0\leq k\leq d(\sigma_0 ;l)-2}{\Diag}
\mhs \left(\gamma_{[l,m;2k+\delta (\sigma_0 ;l),1]}^{(0)}\right)
&\gamma_{[l,m;l,1]}^{(0)}\cdot 
e^{(d(\sigma_0 ;l)-2)}_{d(\sigma_0 ;l)-3}\mmhs 
\end{array}\right).
\end{align*}

When $\sigma_{0,2}=0 $ and $(m,\sigma_{0,1}+l)\in \{0,\pm 2\}\times (1+2\mZ)$,
the matrix $R(\Gamma_{l,m}^0)$ is given by
\begin{align*}
\left(\begin{array}{c}
\mmhs O_{n(\sigma_0 ;l,m),d(\sigma_0 ;l)}\mhs \\[1mm]
R(\Gamma_{l,m}^0)
\end{array}\right)
=&\left(\begin{array}{c}
\mhs \underset{0\leq k\leq d(\sigma_0 ;l)-1}{\Diag}
\mmhs \left(\gamma_{[l,m;2k+\delta (\sigma_0 ;l),-1]}^{(0)}\right) \mhs \\[1mm]
O_{1,d(\sigma_0 ;l)}
\end{array}\right)
+\left(\begin{array}{c}
O_{1,d(\sigma_0 ;l)}\\[1mm]
\mmhs \underset{0\leq k\leq d(\sigma_0 ;l)-1}{\Diag}
\mhs \left(\gamma_{[l,m;2k+\delta (\sigma_0 ;l),0]}^{(0)}\right)\mmhs 
\end{array}\right)\\
&+\left(\begin{array}{cc}
O_{2,d(\sigma_0 ;l)-1}&O_{2,1}\\[1mm]
\mmhs \underset{0\leq k\leq d(\sigma_0 ;l)-2}{\Diag}
\mhs \left(\gamma_{[l,m;2k+\delta (\sigma_0 ;l),1]}^{(0)}\right)
&O_{d(\sigma_0 ;l)-1,1}\mmhs 
\end{array}\right).
\end{align*}

When $\sigma_{0,2}=0 $, $(m,\sigma_{0,1}+l)\in \{\pm 1\}\times (2\mZ)$ 
and $d(\sigma_0 ;l)=1$, 
the matrix $R(\Gamma_{l,m}^0)$ is given by
\begin{align*}
R(\Gamma_{l,m}^0)
=&\left(\gamma_{[l,m;\delta (\sigma_0 ;l),-1]}^{(0)}\right).
\end{align*}

When $\sigma_{0,2}=0 $, $(m,\sigma_{0,1}+l)\in \{\pm 1\}\times (2\mZ)$ 
and $d(\sigma_0 ;l)>1$, 
the matrix $R(\Gamma_{l,m}^0)$ is given by
\begin{align*}
\left(\begin{array}{c}
\mmhs O_{n(\sigma_0 ;l,m),d(\sigma_0 ;l)}\mmhs \\[1mm]
R(\Gamma_{l,m}^0)
\end{array}\right)
=&\underset{0\leq k\leq d(\sigma_0 ;l)-1}{\Diag}
\mhs \left(\gamma_{[l,m;2k+\delta (\sigma_0 ;l),-1]}^{(0)}\right)\\
&+\left(\begin{array}{cc}
O_{1,d(\sigma_0 ;l)-1}&0\\[1mm]
\mmhs \underset{0\leq k\leq d(\sigma_0 ;l)-2}{\Diag}
\mhs \left(\gamma_{[l,m;2k+\delta (\sigma_0 ;l),0]}^{(0)}\right)&
O_{d(\sigma_0 ;l)-1,1}\mmhs 
\end{array}\right)\\
&+\left(\begin{array}{ccc}
O_{2,d(\sigma_0 ;l)-2}&O_{2,1}&O_{2,1}\\[1mm]
\mmhs \underset{0\leq k\leq d(\sigma_0 ;l)-3}{\Diag}
\mhs \left(\gamma_{[l,m;2k+\delta (\sigma_0 ;l),1]}^{(0)}\right)\mmhs &
O_{d(\sigma_0 ;l)-2,1}\mmhs 
&-\gamma_{[l,m;l,1]}^{(0)}\cdot 
e^{(d(\sigma_0 ;l)-3)}_{d(\sigma_0 ;l)-3}\mhs 
\end{array}\right).
\end{align*}

When $\sigma_{0,2}=0 $ and $(m,\sigma_{0,1}+l)\in \{\pm 1\}\times (1+2\mZ)$,
the matrix $R(\Gamma_{l,m}^0)$ is given by
\begin{align*}
\left(\begin{array}{c}
\mmhs O_{n(\sigma_0 ;l,m),d(\sigma_0 ;l)}\mmhs \\[1mm]
R(\Gamma_{l,m}^0)
\end{array}\right)
=&\left(\begin{array}{c}
\mmhs \underset{0\leq k\leq d(\sigma_0 ;l)-1}{\Diag}
\mhs \left(\gamma_{[l,m;2k+\delta (\sigma_0 ;l),-1]}^{(0)}\right)\mhs \\[1mm]
O_{2,d(\sigma_0 ;l)}
\end{array}\right)
+\left(\begin{array}{c}
O_{1,d(\sigma_0 ;l)}\\[1mm]
\mmhs \underset{0\leq k\leq d(\sigma_0 ;l)-1}{\Diag}
\mhs \left(\gamma_{[l,m;2k+\delta (\sigma_0 ;l),0]}^{(0)}\right)\mhs \\[1mm]
O_{1,d(\sigma_0 ;l)}
\end{array}\right)\\
&+\left(\begin{array}{cc}
O_{2,d(\sigma_0 ;l)}\\[1mm]
\mmhs \underset{0\leq k\leq d(\sigma_0 ;l)-1}{\Diag}
\mhs \left(\gamma_{[l,m;2k+\delta (\sigma_0 ;l),1]}^{(0)}\right)\mmhs 
\end{array}\right).
\end{align*}

When $\sigma_{0,2}=1 $, the matrix $R(\Gamma_{l,m}^0)$ is given by
\begin{align*}
\left(\begin{array}{c}
\mhs O_{n(\sigma_0 ;l,m),d(\sigma_0 ;l)}\mhs \\[1mm]
R(\Gamma_{l,m}^0)
\end{array}\right)
=&\left(\begin{array}{c}
\mmhs \underset{0\leq k\leq d(\sigma_0 ;l)-1}{\Diag}
\mhs \left(\gamma_{[l,m;2k+\delta (\sigma_0 ;l),-1]}^{(0)}\right)\mhs \\[1mm]
O_{1,d(\sigma_0 ;l)}
\end{array}\right)
+\left(\begin{array}{c}
O_{1,d(\sigma_0 ;l)}\\[1mm]
\mmhs \underset{0\leq k\leq d(\sigma_0 ;l)-1}{\Diag}
\mhs \left(\gamma_{[l,m;2k+\delta (\sigma_0 ;l),0]}^{(0)}\right)\mhs 
\end{array}\right)\\
&+\mmhs \left(\begin{array}{cc}
O_{2,d(\sigma_0 ;l)-1}&O_{2,1}\\[1mm]
\mmhs \underset{0\leq k\leq d(\sigma_0 ;l)-2}{\Diag}
\mhs \left(\gamma_{[l,m;2k+\delta (\sigma_0 ;l),1]}^{(0)}\right)
&(-1)^{\varepsilon (\sigma_0 ;l+m)}\gamma_{[l,m;l-1,1]}^{(0)}\cdot 
e^{(d(\sigma_0 ;l)-2)}_{d(\sigma_0 ;l)-2}\mhs 
\end{array}\right).
\end{align*}
Here 
\begin{align*}
\gamma_{[l,m;p,1]}^{(0)}
=&(\nu_{0,2}+\rho_{0,2}-l+p)A_{[2l,2m;2l-p+m-2,0]},\\
\gamma_{[l,m;p,0]}^{(0)}
=&-\frac{1}{3}
\Big(2\nu_{0,1}-\nu_{0,2}+2\rho_{0,1}-\rho_{0,2}+lm-3+\frac{m(m+1)}{2}\Big)
A_{[2l,2m;2l-p+m,2]},\\
\gamma_{[l,m;p,-1]}^{(0)}
=&(\nu_{0,2}+\rho_{0,2}+l-p)A_{[2l,2m;2l-p+m+2,4]},\\
n(\sigma_0 ;l,m)
=&\left\{\begin{array}{ll}
(2-m)/2&\text{ if } m\in \{0,\ \pm 2\},\\
(3-m)/2&\text{ if } (m,l+\sigma_{0,2})\in \{\pm 1\}\times (2\mZ ),\\
(1-m)/2&\text{ if } (m,l+\sigma_{0,2})\in \{\pm 1\}\times (1+2\mZ ),
\end{array}\right.
\end{align*}
and $\delta (\sigma_0 ;l)\in \{0,1\}$ such that 
$\delta (\sigma_0 ;l)\equiv l-\sigma_{0,2}\bmod 2$.

In the above equations, we put $A_{[2l,2m;k,i]}=0$ for $k<0$ or $k>2(l+m)$, 
and erase the symbols $\underset{c\leq n\leq c-1}{\Diag}(f(n)),\ 
O_{0,n},\ O_{m,0}$ and $\me_{i}^{(-1)}$. 
\end{thm}
\begin{proof}
Since 
\begin{equation*}
s(l;p,q)(1_3)=\langle 
(v^{(2l)*}_{2l-p}+(-1)^{\varepsilon (\sigma_0 ;l)}v^{(2l)*}_{p}),
v^{(2l)}_{q}
\rangle 
=\delta_{2l-p\hs q}+(-1)^{\varepsilon (\sigma_0 ;l)}\delta_{pq},
\end{equation*}
we have 
\begin{equation}
\label{eqn:fmne}
S(l;p)(1_3)=\me_{2l-p}^{(2l)}+(-1)^{\varepsilon (\sigma_0 ;l)}\me_{p}^{(2l)}.
\end{equation}
Hence $S(l;p)(1_3)\ (p\in Z(\sigma_0 ;l))$ are linearly independent over $\mC$.
Thus we note that it suffices to evaluate 
the both side of the equation (\ref{eqn:statement_thm}) at $1_3\in G$. 

First, we compute $\{\pi_{(\nu_0,\sigma_0)}(X_i)s(l;p,q)\}(1_3)$ 
for $0\leq i\leq 4,\ p\in Z(\sigma_0 ;l)$ and $0\leq q\leq 2l$. 
Since $\{s(l;p,q)\mid 0\leq q\leq 2l\}$ is the standard basis of 
$\langle S(l;p)\rangle $, we obtain
\begin{align*}
\{\pi_{(\nu_0,\sigma_0)}(\sI K_{23})s(l;p,q)\}(1_3)
&=(l-q)s(l;p,q)(1_3)\\
&=(l-q)(\delta_{2l-p\hs q}+(-1)^{\varepsilon (\sigma_0 ;l)}\delta_{pq}),\\
\{\pi_{(\nu_0,\sigma_0)}(K_{13}+\sI K_{12})s(l;p,q)\}(1_3)
&=-q\cdot s(l;p,q-1)(1_3)\\
&=-q(\delta_{2l-p+1\hs q}+(-1)^{\varepsilon (\sigma_0 ;l)}\delta_{p+1\hs q}),\\
\{\pi_{(\nu_0,\sigma_0)}(K_{13}-\sI K_{12})s(l;p,q)\}(1_3)
&=(2l-q)s(l;p,q+1)(1_3)\\
&=(2l-q)(\delta_{2l-p-1\hs q}
+(-1)^{\varepsilon (\sigma_0 ;l)}\delta_{p-1\hs q}).
\end{align*}
Moreover, we obtain 
\begin{align*}
 \{\pi_{(\nu_0,\sigma_0)}(E_\alpha )s(l;p,q)\}(1_3)&=0, 
 && \alpha \in \Sigma^+, \\
 \{\pi_{(\nu_0,\sigma_0)}(H_i)s(l;p,q)\}(1_3)&=
(\nu_{0,i}+\rho_{0,i})s(l;p,q)(1_3)\\
 &=(\nu_{0,i}+\rho_{0,i})
 (\delta_{2l-p\hs q}+(-1)^{\varepsilon (\sigma_0 ;l)}\delta_{pq}), && i=1,2,
\end{align*}
from the definition of principal series representation. 
From these computations and Iwasawa decomposition in Lemma \ref{lem:Iwasawa}, 
we obtain
\begin{align*}
\{\pi_{(\nu_0,\sigma_0)}(X_0)s(l;p,q)\}(1_3)
&=(\nu_{0,2}+\rho_{0,2}+l-q)
(\delta_{2l-p\hs q}+(-1)^{\varepsilon (\sigma_0 ;l)}\delta_{pq}),\\
\{\pi_{(\nu_0,\sigma_0)}(X_1)s(l;p,q)\}(1_3)
&=-\frac{q}{2}(\delta_{2l-p+1\hs q}
+(-1)^{\varepsilon (\sigma_0 ;l)}\delta_{p+1\hs q}),\\
\{\pi_{(\nu_0,\sigma_0)}(X_2)s(l;p,q)\}(1_3)
&=-\frac{1}{3}(2\nu_{0,1}-\nu_{0,2}+2\rho_{0,1}-\rho_{0,2})
(\delta_{2l-p\hs q}+(-1)^{\varepsilon (\sigma_0 ;l)}\delta_{pq}),\\
\{\pi_{(\nu_0,\sigma_0)}(X_3)s(l;p,q)\}(1_3)
&=-\frac{2l-q}{2}(\delta_{2l-p-1\hs q}
-(-1)^{\varepsilon (\sigma_0 ;l)}\delta_{p-1\hs q}),\\
\{\pi_{(\nu_0,\sigma_0)}(X_4)s(l;p,q)\}(1_3)
&=(\nu_{0,2}+\rho_{0,2}-l+q)
(\delta_{2l-p\hs q}+(-1)^{\varepsilon (\sigma_0 ;l)}\delta_{pq}).
\end{align*}
We set 
\[
\pi_{(\nu_0,\sigma_0)}(X_i)S(l;p)
=\sum_{0\leq q\leq 2l}\me_{q}^{(2l)}
\otimes (\pi_{(\nu_0,\sigma_0)}(X_i)s(l;p,q))
.
\]
Then we obtain 
\begin{align*}
\{\pi_{(\nu_0,\sigma_0)}(X_0)S(l;p)\}(1_3)
&=(\nu_{0,2}+\rho_{0,2}-l+p)\me_{2l-p}^{(2l)}
+(-1)^{\varepsilon (\sigma_0 ;l)}(\nu_{0,2}+\rho_{0,2}+l-p)\me_{p}^{(2l)},\\
\{\pi_{(\nu_0,\sigma_0)}(X_1)S(l;p)\}(1_3)
&=-\frac{2l-p+1}{2}\me_{2l-p+1}^{(2l)}
-(-1)^{\varepsilon (\sigma_0 ;l)}\frac{p+1}{2}\me_{p+1}^{(2l)},\\
\{\pi_{(\nu_0,\sigma_0)}(X_2)S(l;p)\}(1_3)
&=-\frac{1}{3}(2\nu_{0,1}-\nu_{0,2}+2\rho_{0,1}-\rho_{0,2})
(\me_{2l-p}^{(2l)}+(-1)^{\varepsilon (\sigma_0 ;l)}\me_{p}^{(2l)}),\\
\{\pi_{(\nu_0,\sigma_0)}(X_3)S(l;p)\}(1_3)
&=-\frac{p+1}{2}\me_{2l-p-1}^{(2l)}
-(-1)^{\varepsilon (\sigma_0 ;l)}\frac{2l-p+1}{2}\me_{p-1}^{(2l)},\\
\{\pi_{(\nu_0,\sigma_0)}(X_4)S(l;p)\}(1_3)
&=(\nu_{0,2}+\rho_{0,2}+l-p)\me_{2l-p}^{(2l)}
+(-1)^{\varepsilon (\sigma_0 ;l)}(\nu_{0,2}+\rho_{0,2}-l+p)\me_{p}^{(2l)}.
\end{align*}

Let us compute $\{\gC_{l,m}S(l;p)\}(1_3)$.
By the above equations, we have
\begin{align*}
\{\gC_{l,m}S(l;p)\}(1_3)
=&\underset{0\leq q\leq 2l}{\sum_{0\leq i\leq 4}}
(R(\iota_i^{(l,m)})\cdot \me_{q}^{(2l)})\otimes 
\{(\pi_{(\nu_0,\sigma_0)}(X_i)s(l;p,q))\}(1_3)\\
=&\sum_{0\leq i\leq 4}R(\iota_i^{(l,m)})\cdot 
\{(\pi_{(\nu_0,\sigma_0)}(X_i)S(l;p))\}(1_3)\\
=&
R(\iota_0^{(l,m)})\cdot \{(\nu_{0,2}+\rho_{0,2}-l+p)\me_{2l-p}^{(2l)}
+(-1)^{\varepsilon (\sigma_0 ;l)}(\nu_{0,2}+\rho_{0,2}+l-p)\me_{p}^{(2l)}\}\\
&+R(\iota_1^{(l,m)})\cdot 
\Big\{-\frac{2l-p+1}{2}\me_{2l-p+1}^{(2l)}
-(-1)^{\varepsilon (\sigma_0 ;l)}\frac{p+1}{2}\me_{p+1}^{(2l)}\Big\}\\
&+R(\iota_2^{(l,m)})\cdot 
\Big\{-\frac{1}{3}(2\nu_{0,1}-\nu_{0,2}+2\rho_{0,1}-\rho_{0,2})
(\me_{2l-p}^{(2l)}+(-1)^{\varepsilon (\sigma_0 ;l)}\me_{p}^{(2l)})\Big\}\\
&+R(\iota_3^{(l,m)})\cdot \Big\{
-\frac{p+1}{2}\me_{2l-p-1}^{(2l)}
-(-1)^{\varepsilon (\sigma_0 ;l)}\frac{2l-p+1}{2}\me_{p-1}^{(2l)}\Big\}\\
&+R(\iota_4^{(l,m)})\hspace{-0.4mm} \cdot \hspace{-0.4mm} \{
(\nu_{0,2}+\rho_{0,2}+l-p)\me_{2l-p}^{(2l)}\mmhs 
+(-1)^{\varepsilon (\sigma_0 ;l)}(\nu_{0,2}+\rho_{0,2}-l+p)\me_{p}^{(2l)}\}.
\end{align*}
Since 
\[
R(\iota_i^{(l,m)})\me_q^{(2l)}=A_{[2l,2m;i+q+m-2,i]}\me_{i+q+m-2}^{(2(l+m))},
\quad -2\leq m\leq 2,
\]
we obtain
\begin{align}
\label{eqn:pf_th_001}
\{\gC_{l,m}S(l;p)\}(1_3)
=&\sum_{-1\leq i\leq 1}
\{
\alpha_{[l,m;p,i]}\me_{2(l+m)-(p+m+2i)}^{(2(l+m))}
+(-1)^{\varepsilon (\sigma_0 ;l)}
\beta_{[l,m;p,i]}\me_{p+m+2i}^{(2(l+m))}
\},
\end{align}
where
\begin{align*}
\alpha_{[l,m;p,1]}
=&(\nu_{0,2}+\rho_{0,2}-l+p)A_{[2l,2m;2l-p+m-2,0]},\\
\alpha_{[l,m;p,0]}
=&-\frac{1}{3}(2\nu_{0,1}-\nu_{0,2}+2\rho_{0,1}-\rho_{0,2})
A_{[2l,2m;2l-p+m,2]}\\
&-\frac{2l-p+1}{2}A_{[2l,2m;2l-p+m,1]}-\frac{p+1}{2}A_{[2l,2m;2l-p+m,3]},\\
\alpha_{[l,m;p,-1]}
=&(\nu_{0,2}+\rho_{0,2}+l-p)A_{[2l,2m;2l-p+m+2,4]},\\
\beta_{[l,m;p,1]}
=&(\nu_{0,2}+\rho_{0,2}-l+p)A_{[2l,2m;p+m+2,4]},\\
\beta_{[l,m;p,0]}
=&-\frac{1}{3}(2\nu_{0,1}-\nu_{0,2}+2\rho_{0,1}-\rho_{0,2})A_{[2l,2m;p+m,2]}\\
&-\frac{p+1}{2}A_{[2l,2m;p+m,1]}-\frac{2l-p+1}{2}A_{[2l,2m;p+m,3]},\\
\beta_{[l,m;p,-1]}
=&(\nu_{0,2}+\rho_{0,2}+l-p)A_{[2l,2m;p+m-2,0]}.
\end{align*}
By the relations of the coefficients $A_{[2l,2m;k,i]}$ in Lemma 
\ref{lem:rel_clebsch}, we see that 
\[
\alpha_{[l,m;p,i]}=(-1)^m\beta_{[l,m;p,i]}=\gamma_{[l,m;p,i]}^{(0)},\quad 
-1\leq i\leq 1.
\]
Therefore, (\ref{eqn:pf_th_001}) become
\begin{align}
\{\gC_{l,m}S(l;p)\}(1_3)
=&\sum_{-1\leq i\leq 1}
\gamma_{[l,m;p,i]}^{(0)}\{
\me_{2(l+m)-(p+m+2i)}^{(2(l+m))}
+(-1)^{\varepsilon (\sigma_0 ;l)+m}
\me_{p+m+2i}^{(2(l+m))}
\}.\label{eqn:pf_th_002}
\end{align}
From the equations (\ref{eqn:fmne}), (\ref{eqn:pf_th_002}) and 
\[
\varepsilon (\sigma_0 ;l)+m\equiv 
\varepsilon (\sigma_0 ;l+m)\bmod 2,
\]
we obtain the assertion.
\end{proof}

\section{The $(\g ,K)$-module structures of 
the generalized principal series representations}
\label{sec:structure2}
In this section, we set $i=1$ or $2$.

\subsection{Discrete series representations of $SL^\pm (2,\mR )$}

We set $y_0=\diag (1,-1)\in O(2)$. Then  a discrete series representation 
$(D_k,V_{D_k})$ of $SL^\pm (2,\mR )$ is uniquely determined by 
specifying the $G'=SL(2,\mR )$-module structure together with the action of $y_0$. 
Since $D_k|_{G'}=D_k^+\oplus D_k^-$ and $D_k^+\oplus D_k^-$ is identified 
with $G'$-submodule of the principal series representation 
$(\pi_{(\nu ,\sigma )} ,H_{(\nu ,\sigma )} )$ of $G'$ by 
Proposition \ref{prop:discrete}, we obtain the following realization 
of $(D_k,V_{D_k})$:
\begin{align*}
V_{D_k,O(2)}=&\bigoplus_{\alpha \in \mZ_{\geq 0}}W_{k+2\alpha }&
\big(W_p=&\mC \cdot \chi_p+\mC \cdot 
\chi_{-p}\big)
\end{align*}
and 
\begin{align*}
D_k(w)\chi_{p}&=\sI p\chi_{p},&
D_k(x_+)\chi_{p}&=(k+p)\chi_{p+2},&
D_k(x_-)\chi_{p}&=(k-p)\chi_{p-2},\\
D_k(\kappa_t)\chi_{p}&=e^{\sI pt}\chi_{p}& (t\in \mR),&&
D_k(y_0)\chi_{p}=&\chi_{-p}.
\end{align*}
Here we denote differential of $D_k$ again by $D_k$ and 
the $O(2)$-finite part of $V_{D_k}$ by $V_{D_k,O(2)}$.

\subsection{Irreducible decompositions of 
$(\pi_{(\nu_1,\sigma_1)}|_K,H_{(\nu_1,\sigma_1)})$ 
and $(\pi_{(\nu_2,\sigma_2)}|_K,H_{(\nu_2,\sigma_2)})$ 
as $K$-modules}
\label{subsec:peter-weyl12}

We analyzes the $K$-type of the representation space 
$H_{(\nu_i,\sigma_i)}$ of the $P_i$-principal series 
representation. 
the target space $V_{\sigma_i}$ of functions 
$\mathbf{f}$ in $H_{(\nu_i,\sigma_i)}$ has a decomposition:
\[
V_{\sigma_i}=V_{D_k}=
\hatoplus_{\alpha \in \mZ_{\geq 0}}W_{k+2\alpha }.
\]
Denote the corresponding decomposition of $\mathbf{f}$ by
\[
\mathbf{f}(x)=\sum_{\alpha =0}^{\infty}
(f_{k+2\alpha }(x)\otimes \chi_{k+2\alpha }
+f_{-(k+2\alpha )}(x)\otimes \chi_{-(k+2\alpha )}).
\]
From the definition of the space $H_{(\nu_i,\sigma_i)}$, 
we have
\begin{align*}
\mathbf{f}|_K(mx)&=\sigma_i(m)\mathbf{f}|_K(x)&
(\text{a.e. } x\in K,\ m\in K_i=M_i\cap K\simeq O(2)).
\end{align*}
For $m=m_i(\kappa_t),\ m_i(y_0)$, comparing the coefficients of $\chi_p$ 
in the left hand side with those in the right hand side, we have the equations
\begin{align*}
f_p|_K(m_i(\kappa_t)x)&=e^{\sI pt}f_p|_K(x),&
f_p|_K(m_i(y_0)x)&=f_{-p}|_K(x).
\end{align*}

Moreover, from the equality of inner products
\[
\int_K\| \mathbf{f}|_K(x)\|_{\sigma_i}^2 dx
=\sum_{\varepsilon \in \{\pm 1\} ,\ \alpha \in \mZ_{\geq 0}}
\left\{
\int_K\left| f_{\varepsilon (k+2\alpha )}|_K(x)\right| dx
\right\}
\|\chi_{\varepsilon (k+2\alpha )}\|_{\sigma_i}^2,
\]
we have $f_p|_K\in L^2(K)$. 
Therefore $\mathbf{f}|_K$ belongs to
\[
\hatoplus_{\alpha \in \mZ_{\geq 0}}L^2_i(K;W_{k+2\alpha})
\]
where
\begin{align*}
&L^2_i(K;W_p)
=\{\mathbf{f}\colon K\to W_p\mid 
\mathbf{f}(x)=f(x)\otimes \chi_p+f(m_i(y_0)x)\otimes \chi_{-p},\ 
f\in L^2_{(K_i^\circ ,\chi_p)}(K), x\in K\},\\
&L^2_{(K_i^\circ ,\chi_p)}(K)=\{f\in L^2(K)\mid 
f(m_i(\kappa_t)x)=e^{\sI pt}f(x),\ m_i(\kappa_t)\in K_i^\circ ,\ x\in K\}.
\end{align*}
Here $K_i^\circ $ means the connected component of $K_i$, 
which is isomorphic to $SO(2)$. 
We easily see that the restriction map
\[
r_K^{(i)}\colon H_{(\nu_i,\sigma_i)}\ni \mathbf{f}\mapsto 
\mathbf{f}|_K\in \hatoplus_{\alpha \in \mZ_{\geq 0}}L^2_i(K;W_{k+2\alpha})
\]
is a $K$-isomorphism. 

By the Peter-Weyl's theorem, 
we have an irreducible decomposition of 
$L^2_{(K_i^\circ ,\chi_p)}(K)$: 
\[
L^2_{(K_i^\circ ,\chi_p)}(K)\simeq \hatoplus_{l \in \mZ_{\geq 0}} 
(V_{2l}^{*}[\xi_{(i;-p)} ])\otimes_\mC V_{2l} .
\]
Here 
\[
\xi_{(i;p)}\colon K_i^\circ \ni m_i(\kappa_t)\mapsto e^{\sI pt}\in \mC^\times
\]
and 
$V[\xi_{(i;p)}]$ means the 
$\xi_{(i;p)}$-isotypic component 
in $(\tau|_{K_i^\circ },V)$ for a $K$-module $(\tau ,V)$.

In this section, we denote by $\{v_{1,q}^{(2l)}\mid 0\leq q\leq 2l\}$ 
the standard basis of $V_{2l}$. 
We define an another basis 
$\{v_{2,q}^{(2l)}\mid 0\leq q\leq 2l\}$ of $V_{2l}$ by
\[
v_{2,q}^{(2l)}=\tau_{2l}(u_c)v_{1,q}^{(2l)}
=\frac{1}{2^l}(x+y)^q(-x+y)^{2l-q}
\quad (0\leq q\leq 2l)
\] 
where 
\[
u_c=
\left(\begin{array}{ccc}
0&0&-1\\
0&1&0\\
1&0&0
\end{array}\right)\in SO(3).
\] 
We note that 
$v\in V_{2l}[\xi_{(i;-p)} ]$ if and only if 
\begin{align*}
\tau_{2l}(m_i(\kappa_t))v=&\xi_{(i;-p)}(m_i(\kappa_t))v=e^{-\sI pt}v&
(t\in \mR )
\end{align*}
for $v\in V_{2l}$.
From the definition of $(\tau_{2l},V_{2l})$ and 
\begin{align*}
\varphi^{-1}(m_1(\kappa_t))=\varphi^{-1}(u_c^{-1}m_2(\kappa_t)u_c)
=&\left\{\pm \diag (e^{-\sI t/2},e^{\sI t/2})\right\},
\end{align*}
we have $\tau_{2l}(m_i(\kappa_t))v^{(2l)}_{i,q}
=e^{\sI (q-l)t}v^{(2l)}_{i,q}$. 
Hence we have 
\begin{align*}
V_{2l}[\xi_{(i;-p)}]=&
\left\{\begin{array}{ll}
\mC \cdot v^{(2l)}_{i,l-p}&\text{ if }-l\leq p\leq l,\\
0&\text{ otherwise }.
\end{array}\right.
\end{align*}
By the identification $V_{2l}^{*}=V_{2l}$ in 
Lemma \ref{lem:dual_Krep}, we obtain
\[
L^2_{(K_i^\circ ,\chi_p)}(K)\simeq 
\underset{-l\leq p\leq l}{\hatoplus_{l \in \mZ_{\geq 0}} }
(\mC \cdot v^{(2l)*}_{i,l+p})\otimes_\mC V_{2l}.
\]
Moreover, since 
\begin{align*}
\varphi^{-1}(m_1(y_0))&=\left\{\pm 
\left(\begin{array}{cc}
0&1\\
-1&0
\end{array}\right)
\right\},&
\varphi^{-1}(u_c^{-1}m_2(y_0)u_c)&=
\left\{\pm 
\left(\begin{array}{cc}
0&\sI\\
\sI&0
\end{array}\right)
\right\},
\end{align*}
we have 
\begin{align*}
\tau_{2l}^*(m_1(y_0)^{-1})v^{(2l)*}_{1,l+p}&=(-1)^{l+p}v^{(2l)*}_{1,l-p},&
\tau_{2l}^*(m_2(y_0)^{-1})v^{(2l)*}_{2,l+p}&=(-1)^{l}v^{(2l)*}_{2,l-p}.
\end{align*}
For $0\leq p\leq l-k$ such that $p\equiv l-k\bmod 2$, 
we define the elementary function 
$t_i(l;p,q)\in H_{(\nu_i,\sigma_i)}$ by 
\[
t_i(l;p,q)=r_K^{(i)-1}(\tilde{t}_i(l;p,q))
\]
where
\begin{align*}
\tilde{t}_1(l;p,q)(x)=&
\langle v^{(2l)*}_{1,2l-p},\tau_{2l}(x)v^{(2l)}_{1,q}\rangle \otimes \chi_{l-p}
+(-1)^{p}\langle v^{(2l)*}_{1,p},
\tau_{2l}(x)v^{(2l)}_{1,q}\rangle \otimes \chi_{p-l},\\
\tilde{t}_2(l;p,q)(x)=&
\langle v^{(2l)*}_{2,2l-p},\tau_{2l}(x)v^{(2l)}_{1,q}\rangle \otimes \chi_{l-p}
+(-1)^{l}\langle v^{(2l)*}_{2,p},
\tau_{2l}(x)v^{(2l)}_{1,q}\rangle \otimes \chi_{p-l}.
\end{align*}

Let $T_i(l;p)$ be a column vector of degree 
$2l+1$ with its $q+1$-th component $t_i(l;p,q)$, i.e. 
${}^t(\ t_i(l;p,0),\ t_i(l;p,1),\ \cdots \ ,t_i(l;p,2l)\ ).$

Moreover we denote by $\langle T_i(l;p)\rangle $ 
the subspace of $H_{(\nu_i,\sigma_i)}$ generated by 
the functions in the entries 
of the vector $T_i(l;p)$, i.e. 
$\langle T_i(l;p)\rangle =
\bigoplus_{q=0}^{2l}\mC \cdot t_i(l;p,q)\simeq V_{2l}$. 
Via the unique isomorphism between $\langle T_i(l;p)\rangle$ 
and $V_{2l}$, we identify $\{t_i(l;p,q)\mid 0\leq q\leq 2l\}$ 
with the standard basis. 

From above arguments, we obtain the following. 
\begin{prop}
As an unitary representation of $K$, it has an irreducible 
decomposition:
\[
H_{(\nu_i,\sigma_i)}
=\underset{p\equiv l-k\bmod 2}
{\hatoplus_{l\in \mZ_{\geq 0},\ 0\leq p\leq l-k}}
\langle T_i(l;p )\rangle
\]
for $i=1,2$. 
Then the $\tau_{2l}$-isotypic component of 
$\pi_{(\nu_i,\sigma_i)}$ is given by 
\[
\underset{p\equiv l-k\bmod 2}{\bigoplus_{0\leq p\leq l-k}}
\langle T_i(l;p ) \rangle .
\]
\end{prop}
\begin{cor}
Let $d(\sigma_i ;l)$ be the dimension of the space 
$\Hom_K(V_{2l} ,H_{(\nu_i,\sigma_i),K})$ of intertwining operators. 
Then
\[
d(\sigma_i ;l)=\left\{
\begin{array}{ll}
(l-k+2)/2&\text{if $k\leq l$ and $l-k$ is even,}\\
(l-k+1)/2&\text{if $k\leq l$ and $l-k$ is odd,}\\
0&\text{if $k> l$.}
\end{array}
\right.
\]
\end{cor}

\subsection{The canonical blocks of elementary functions}
\label{subsec:canonical_blocks12}

By the identification introduced in 
Subsection \ref{subsec:canonical_blocks}, 
we identify $T_i(l;p)$ with
the injective $K$-homomorphism 
\[
V_{2l} \ni v^{(2l)}_{1,q}\mapsto t_i(l;p,q)
\in H_{(\nu_i,\sigma_i),K},\quad 0\leq q\leq 2l
\]
for $0\leq p \leq l-k$ such that $p\equiv l-k \bmod 2$. 
We note that $\{ T_i(l;p) \mid 0\leq p \leq l-k,\ p\equiv l-k \bmod 2\} $ 
is a basis of $\Hom_K (V_{2l},H_{(\nu_i,\sigma_i),K})$ and we call 
it \textit{the induced basis from the standard basis}.

We define a certain matrix of elementary functions corresponding to 
the induced basis 
$\{ T_i(l;p) \mid 0\leq p \leq l-k,\ p\equiv l-k \bmod 2\} $ 
of $\Hom_K (V_{2l},H_{(\nu_i,\sigma_i),K})$
for each $K$-type $\tau_{2l}$ of 
our $P_i$-principal series representation $\pi_{(\nu_i,\sigma_i)}$.
\begin{defn}
\textit{
For $l\in \mZ_{\geq 0}$ such that $d(\sigma_i ;l)>0$, the following 
$(2l+1)\times d(\sigma_i ;l)$ matrix $\mT_i (\sigma_i ;l)$ is called
\textit{the canonical block of elementary functions} 
for $\tau_{2l}$-isotypic component:
When $l-k$ is even, we consider the matrix
\begin{align*}
\mT_i(\sigma_i ;l)=&
(\ T_i(l;0),\ T_i(l;2),\ T_i(l;4),\ \cdots \ ,
T_i(l;l-k)\ ).
\end{align*}
When $l-k$ is odd, we consider the matrix
\begin{align*}
\mT_i(\sigma_i ;l)=&
(\ T_i(l;1),\ T_i(l;3),\ T_i(l;5),\ \cdots \ ,
T_i(l;l-k)\ ).
\end{align*}
}
\end{defn}

\subsection{The contiguous relations}
\label{subsec:matrix_rep}

\begin{lem}
\label{lem:notIwasawa}
\textit{
(i) The standard basis $\{X_{j}\mid 0\leq j\leq 4\}$ of $\gp_\mC$ 
have the following expressions according to the decomposition 
$\g_\mC =\gn_{1,\mC} \oplus \ga_{1,\mC} \oplus \gm_{1,\mC} \oplus \gk_\mC $:
\begin{align*}
X_{0}=&m_1(x_-),&
X_{1}=&-(E_{e_1-e_3}+\sI E_{e_1-e_2})
	+\frac{1}{2}(K_{13}+\sI K_{12}),\\
X_{2}=&-\frac{1}{3}H^{(1)},&
X_{3}=&(E_{e_1-e_3}-\sI E_{e_1-e_2})
	-\frac{1}{2}(K_{13}-\sI K_{12}),&
X_{4}=&m_1(x_+).
\end{align*}
(ii) The basis 
$\{X_{j}'=u_cX_{j}u_c^{-1}\mid 0\leq j\leq 4\}$ of $\gp_\mC$ 
have the following expressions according to the decomposition 
$\g_\mC =\gn_{2,\mC} \oplus \ga_{2,\mC} \oplus \gm_{2,\mC} \oplus \gk_\mC $:
\begin{align*}
X_{0}'=&-m_2(x_-),&
X_{1}'=&(E_{e_1-e_3}-\sI E_{e_2-e_3})
	-\frac{1}{2}(K_{13}-\sI K_{23}),\\
X_{2}'=&\frac{1}{3}H^{(2)},&
X_{3}'=&-(E_{e_1-e_3}+\sI E_{e_2-e_3})
	+\frac{1}{2}(K_{13}+\sI K_{23}),&
X_{4}'=&-m_2(x_+),
\end{align*}
}
\end{lem}
\begin{proof}
We obtain the assertion immediately from Lemma \ref{lem:ad_K_p}. 
\end{proof}

We give the matrix representation of $\Gamma_{l,m}^i$ 
with respect to the induced basis as follows.

\begin{thm}
\label{th:main2}
For $i=1,2$ and $-2\leq m\leq 2$, we have an following equation 
with the matrix representation 
$R(\Gamma_{l,m}^i)\in M_{d(\sigma_i ;l+m),d(\sigma_i ;l)}(\mC )$ 
of $\Gamma_{l,m}^i$ with respect to the induced basis 
$\{ T_i(l;p) \mid 0\leq p\leq l-k,\ p\equiv l-k\bmod 2\} $: 
\begin{equation}
\gC_{l,m}\mT_i (\sigma_i ;l)
=\mT_i (\sigma_i ;l+m)\cdot R(\Gamma_{l,m}^i).
\label{eqn:statement_thm12}
\end{equation}
Explicit expressions of the matrix $R(\Gamma_{l,m}^i)$ of size 
$d(\sigma_i ; l+m)\times d(\sigma_i ;l)$ is given as follows:

The matrix $R(\Gamma_{l,m}^i)$ is given by
\begin{align*}
\left(\begin{array}{c}
\mmhs O_{n(\sigma_i ;l,m),d(\sigma_i ;l)}\mmhs \\[1mm]
R(\Gamma_{l,m}^i)
\end{array}\right) 
=&\left(\begin{array}{c}
\mmhs \underset{0\leq j\leq d(\sigma_i ;l)-1}{\Diag}
\mhs \left(\gamma_{[l,m;2j+\delta (\sigma_i ;l),-1]}^{(i)}\right)\mhs \\[1mm]
O_{1,d(\sigma_i ;l)}
\end{array}\right)
+\left(\begin{array}{c}
O_{1,d(\sigma_i ;l)}\\[1mm]
\mmhs \underset{0\leq j\leq d(\sigma_i ;l)-1}{\Diag}
\mhs \left(\gamma_{[l,m;2j+\delta (\sigma_i ;l),0]}^{(i)}\right)\mhs 
\end{array}\right)\\
&+\left(\begin{array}{cc}
O_{2,d(\sigma_i ;l)-1}&O_{2,1}\\
\mmhs \underset{0\leq j\leq d(\sigma_i ;l)-2}{\Diag}
\mhs \left(\gamma_{[l,m;2j+\delta (\sigma_i ;l),1]}^{(i)}\right)&
O_{d(\sigma_i ;l)-1,1}
\end{array}\right).
\end{align*}

Here 
\begin{align*}
\gamma_{[l,m;p,1]}^{(i)}
=&(-1)^{i+1}(k-l+p)A_{[2l,2m;2l-p+m-2,0]},\\
\gamma_{[l,m;p,0]}^{(i)}
=&\frac{(-1)^{i}}{3}
\Big(\nu_i+\rho_i+lm-3+\frac{m(m+1)}{2}\Big)
A_{[2l,2m;2l-p+m,2]},\\
\gamma_{[l,m;p,-1]}^{(i)}
=&(-1)^{i+1}(k+l-p)A_{[2l,2m;2l-p+m+2,4]},\\
n(\sigma_i ;l,m)
=&\left\{\begin{array}{ll}
(2-m)/2&\text{ if } m\in \{0,\ \pm 2\},\\
(3-m)/2&\text{ if } (m,l-k)\in \{\pm 1\}\times (2\mZ ),\\
(1-m)/2&\text{ if } (m,l-k)\in \{\pm 1\}\times (1+2\mZ ),
\end{array}\right.
\end{align*}
and $\delta (\sigma_i ;l)\in \{0,1\}$ such that 
$\delta (\sigma_i ;l)\equiv l-k\bmod 2$.

In the above equations, we put $A_{[2l,2m;p,j]}=0$ for $p<0$ or $p>2(l+m)$, 
and erase the symbols 
$\underset{c_0\leq n\leq c_1}{\Diag}(f(n))\ (c_0>c_1),\ 
O_{m,n}\ (m\leq 0\text{ or }n\leq 0)$.
\end{thm}

\begin{proof}
By the similarly computation in the proof of Theorem \ref{th:main} 
using Lemma \ref{lem:notIwasawa} (i), 
we obtain the assertion in the case of $i=1$. 
However, in the case of $i=2$, 
It is difficult to prove the assertion by the same method 
since the value of $T_2(l;p)$ at $1_3\in G$ is not simple. 
We avoid this problem as follows.

We put 
\begin{align*}
t_2'(l;p,j)&=\pi_{(\nu_2,\sigma_2)}(u_c)t_2(l;p,j)
\quad (0\leq j\leq 2l),\\
T_2'(l;p)&={}^t(\ t_2'(l;p,0),\ t_2'(l;p,1),\ \cdots \ ,t_2'(l;p,2l)\ ),\\
\mT_2'(\sigma_2 ;l)&=
\left\{\begin{array}{ll}
(\ T_2'(l;0),\ T_2'(l;2),\ T_2'(l;4),\ \cdots \ ,T_2'(l;l-k)\ )
&\text{if $l-k$ is even,}\\
(\ T_2'(l;1),\ T_2'(l;3),\ T_2'(l;5),\ \cdots \ ,T_2'(l;l-k)\ )
&\text{if $l-k$ is odd,}
\end{array}\right. \\
\gC'_{l,m}&=\sum_{j=0}^4R(\iota_j^{(l,m)})\otimes X_j'.
\end{align*}
Then we see that 
\begin{align}
\gC'_{l,m}\mT_2' (\sigma_2 ;l)
=&\mT_2' (\sigma_2 ;l+m)\cdot R(\Gamma_{l,m}^2),\label{eqn:pf_th2_001}
\end{align}
and
\begin{align*}
T_2'(l;p)(1_3)=&
\me_{2l-p}^{(2l)} \otimes \chi_{l-p}+(-1)^{l}\me_{p}^{(2l)} \otimes \chi_{p-l}.
\end{align*}
Thus, by the similarly computation as in Lemma \ref{lem:notIwasawa} (ii),  
we also obtain the assertion in the case of $i=2$ evaluating 
the both side of the equation (\ref{eqn:pf_th2_001}) at $1_3\in G$. 
\end{proof}

\section{The action of $\gp_\mC$}
\label{sec:examples}
The linear map $\Gamma_{l,m}^i$ characterize the action of $\gp_\mC$. 
In this section, we give a explicit description of 
the action of $\gp_\mC$ on the elementary functions. 

\subsection{The projectors for $V_{l}\otimes_\mC V_4$}

For $-2\leq m\leq 2$, we describe a surjective $\gs \gl (2,\mC )$-homomorphism 
$P^{l}_{2m}$ from $V_{l}\otimes_\mC V_4$ to $V_{l+2m}$
 in terms of the standard basis as follows.

\begin{lem}
\label{lem:projector}
\textit 
Let $\{v_{q}^{(l)}\mid 0\leq q\leq l\} $ be the standard basis of $V_l$ 
for $l\in \mZ_{\geq 0}$. 
We put $v_q^{(l)}=0$ when $q<0$ or $q>l$. 

We define linear maps 
$P^{l}_{2m}\colon V_l\otimes_\mC V_4\to V_{l+2m}\ (-2\leq m \leq 2)$ by 
\[
P^{l}_{2m}(v_q^{(l)}\otimes w_r)
=B_{[l,2m;q,r]}\cdot v_{q+r+m-2}^{(l+2m)},
\]
when $V_{l+2m}$-component of $V_l\otimes_\mC V_4$ does not vanish. 

Here the coefficients $B_{[l,2m;q,r]}=b(l,2m;q,r)/d'(l,2m)$ are defined by 
following formulae. 

\noindent {\bf Formula 1:} 
The coefficients of 
$P^{l}_{4}\colon V_l\otimes_\mC V_4\to V_{l+4}$ 
are given as follows:
\begin{align*}
b(l,4;q,r)=&1\quad (0\leq r\leq 4),&d'(l,4)=&1.
\end{align*}
\noindent {\bf Formula 2:} 
The coefficients of 
$P^{l}_{2}\colon V_l\otimes_\mC V_4\to V_{l+2}$ are given as follows:
\begin{align*}
b(l,2;q,0)=&4q,&
b(l,2;q,1)=&-(l-4q),&
b(l,2;q,2)=&-2(l-2q),\\
b(l,2;q,3)=&-(3l-4q),&
b(l,2;q,4)=&-4(l-q),&
d'(l,2)=&l+4.
\end{align*}
\noindent {\bf Formula 3:} 
The coefficients of 
$P^{l}_{0}\colon V_l\otimes_\mC V_4 \to V_{l}$ are given as follows:
\begin{align*}
b(l,0;q,0)=&6q(q-1),&
b(l,0;q,1)=&-3q(l-2q+1),\\
b(l,0;q,2)=&l^2-6lq+6q^2-l,&
b(l,0;q,3)=&3(l-2q-1)(l-q),\\
b(l,0;q,4)=&6(l-q)(l-q-1),&
d'(l,0)=&(l+3)(l+2).
\end{align*}
\noindent {\bf Formula 4: }
The coefficients of 
$I^{l}_{-2}\colon  V_{l-2}\to V_l\otimes_\mC V_4$ are given as follows:
\begin{align*}
b(l,-2;q,0)=&4q(q-1)(q-2),&
b(l,-2;q,1)=&-q(q-1)(3l-4q+2),\\
b(l,-2;q,2)=&2q(l-2q)(l-q),&
b(l,-2;q,3)=&-(l-4q-2)(l-q)(l-q-1),\\
b(l,-2;q,4)=&-4(l-q)(l-q-1)(l-q-2),&
d'(l,-2)=&(l+2)(l+1)l.
\end{align*}
\noindent {\bf Formula 5:} 
The coefficients of 
$I^{l}_{-4}\colon  V_{l-4}\to V_l\otimes_\mC V_4$ are given as follows:
\begin{align*}
b(l,-4;q,0)=&q(q-1)(q-2)(q-3),&
b(l,-4;q,1)=&-q(q-1)(q-2)(l-q),\\
b(l,-4;q,2)=&q(q-1)(l-q)(l-q-1),&
b(l,-4;q,3)=&-q(l-q)(l-q-1)(l-q-2),\\
b(l,-4;q,4)=&(l-q)(l-q-1)(l-q-2)(l-q-3),\hspace{-12mm}& 
&d'(l,-4)=(l+1)l(l-1)(l-2).
\end{align*}

Then $P^{l}_{2m}$ is the generator of 
$\Hom_{\gs \gl (2,\mC )} (V_l\otimes_\mC V_4,V_{l+2m})$ 
such that $P^{l}_{2m}\circ I^{l}_{2m}=\id_{V_{l+2m}}$.
\end{lem}
\begin{proof}
The composite
\[
V_l\otimes_\mC V_4
\simeq V_l^*\otimes_\mC V_4^*
\simeq (V_l\otimes_\mC V_4)^*\ni f\mapsto f\circ I^{l}_{2m}\in 
V_{l+2m}^*\simeq V_{l+2m}
\]
is a surjective $\gs \gl (2,\mC )$-homomorphism from 
$V_l\otimes_\mC V_4$ to $V_{l+2m}$, which is unique up to scalar multiple. 
Therefore we obtain the assertion from Proposition \ref{prop:injector} 
and Lemma \ref{lem:dual_Krep}. 
\end{proof}

\subsection{The action of $\gp_\mC$ on the elementary functions}

\begin{prop}
\label{prop:p_action}
(i) An explicit expression of the action of $\gp_\mC$ on the basis 
$\{s(l;p,q)\mid l\geq 0,\ p\in Z(\sigma_0;l),\ 0\leq q\leq 2l\}$ 
of $H_{(\nu_0,\sigma_0),K}$ is given by following equation:
\begin{align*}
\pi_{(\nu_0,\sigma_0)}(X_r)s(l;p,q)
&=\underset{-2\leq m\leq 2}{\sum_{-1\leq j\leq 1}}
\gamma_{[l,m;p,j]}^{(0)}B_{[2l,2m;q,r]}
s(l+m;p+m+2j,q+m+r-2).
\end{align*}
Here we put
\begin{align*}
\gamma_{[0,m;0,j]}^{(0)}&=B_{[0,2m;0,r]}=0\text{ for }m<2,\quad 
\gamma_{[1,m;p,j]}^{(0)}=B_{[2,2m;q,r]}=0\text{ for }m<0,\\
s(l;p,q)&=0\text{ whenever } p\leq l \text{ such that }p\notin Z(\sigma_0;l) 
\text{ or }q<0\text{ or }q>2l,\\
s(l;p,q)&=(-1)^{\varepsilon (\sigma_0 ;l)}s(l;2l-p,q)
\text{ for }p>l.
\end{align*}
(ii) For $i=1,2$, the explicit expression of the action of $\gp_\mC$ 
on the basis 
$\{ t_i(l;p,q) \mid l\geq k,\ 0\leq p \leq l-k,\ p\equiv l-k \bmod 2
,\ 0\leq q\leq 2l\} $
of $H_{(\nu_i,\sigma_i),K}$ is given by following equation:
\begin{align*}
\pi_{(\nu_i,\sigma_i)}(X_r)t_i(l;p,q)
&=\underset{-2\leq m\leq 2}{\sum_{-1\leq j\leq 1}}
\gamma_{[l,m;p,j]}^{(i)}B_{[2l,2m;q,r]}
t_i(l+m;p+m+2j,q+m+r-2)
\end{align*}
Here we put 
$t_i(l;p,q)=0$ 
unless $0\leq p \leq l-k,\ p\equiv l-k \bmod 2$ and $0\leq q\leq 2l$.
\end{prop}

\begin{proof}
Since
\begin{align*}
\pi_{(\nu_0,\sigma_0)}(X_r)s(l;p,q)
=&\sum_{-2\leq m\leq 2}
\Gamma_{l,m}^0(S(l;p))\circ P^{l}_{2m}(v^{(2l)}_q\otimes X_r),\\
\pi_{(\nu_i,\sigma_i)}(X_r)t_i(l;p,q)
=&\sum_{-2\leq m\leq 2}
\Gamma_{l,m}^i(T_i(l;p))\circ P^{l}_{2m}(v^{(2l)}_q\otimes X_r)&
(i=1,2),
\end{align*}
we obtain the assertion from Theorem \ref{th:main}, \ref{th:main2} and 
Lemma \ref{lem:projector}.
\end{proof}

\def\cprime{$'$}


\begin{thebibliography}{1}

\bibitem{MR0330355}
Hrvoje Kraljevi{\'c}.
\newblock Representations of the universal convering group of the group {${\rm
  SU}(n,\,1)$}.
\newblock {\em Glasnik Mat. Ser. III}, Vol. 8(28), pp. 23--72, 1973.

\bibitem{MR2070374}
Hiroyuki Manabe, Taku Ishii, and Takayuki Oda.
\newblock Principal series {W}hittaker functions on {${\rm SL}(3,R)$}.
\newblock {\em Japan. J. Math. (N.S.)}, Vol.~30, No.~1, pp. 183--226, 2004.

\bibitem{master_2_2006}
Tadashi Miyazaki.
\newblock The $(\mathfrak{g},k)$-module structures of principal series
  representations of {${\rm Sp}(2,\bold R)$}, preprint.

\bibitem{pre_standard_1_2006}
Takayuki Oda.
\newblock The standard $(\mathfrak{g},k)$-modules of {${\rm Sp}(2,\bold R)$\
  I}, preprint.

\bibitem{pre_standard_2}
Takayuki Oda.
\newblock The standard $(\mathfrak{g},k)$-modules of {${\rm Sp}(2,\bold R)$\
  II}, preprint.

\bibitem{MR0453925}
Ernest Thieleker.
\newblock On the integrable and square-integrable representations of {${\rm
  Spin}(1, 2m)$}.
\newblock {\em Trans. Amer. Math. Soc.}, Vol. 230, pp. 1--40, 1977.

\end{thebibliography}

\end{document}